\documentclass[12pt,intlimits,sumlimits, reqno]{amsart}
\usepackage{amsaddr}
\usepackage{amsthm,amsfonts,amssymb}
\usepackage[mathcal]{euscript}
\usepackage[scr]{rsfso}
\usepackage{graphicx}
\usepackage{indentfirst}
\usepackage{mathtools}
\usepackage[hidelinks,draft=false]{hyperref}
\usepackage{xcolor}
\usepackage{etoolbox}
\usepackage{subcaption}
\usepackage{tikz}

\usepackage{orcidlink}
\mathtoolsset{showonlyrefs=true}

\usetikzlibrary{decorations.pathreplacing}
\usetikzlibrary{arrows,decorations.markings,intersections}
\usetikzlibrary{shapes.misc}
\usetikzlibrary{patterns}
\tikzset{cross/.style={cross out, draw=black, minimum size=2*(#1-\pgflinewidth), inner sep=0pt, outer sep=0pt}, cross/.default={2pt}}
\makeatletter
\renewcommand\section{\@startsection{section}{1}%
  \z@{-.5\linespacing\@plus-.7\linespacing}{.5\linespacing}%
  {\normalfont\scshape\centering}}

\renewcommand\subsection{\@startsection{subsection}{2}%
  \z@{-.5\linespacing\@plus-.7\linespacing}{.5\linespacing}%
  {\normalfont\bfseries}}
\renewcommand\subsubsection{\@startsection{subsubsection}{3}%
  \z@{.5\linespacing\@plus.7\linespacing}{-.5em}%
  {\normalfont\itshape}}
\makeatother

\usepackage[a4paper,showframe=false,headheight=0cm, margin=2cm]{geometry}
\newtheorem{thm}{Theorem}[section]

\newtheorem{cor}[thm]{Corollary}
\newtheorem{lem}[thm]{Lemma}
\newtheorem{prop}[thm]{Proposition}
\theoremstyle{definition}

\newtheorem{remark}{Remark}
\numberwithin{equation}{section}

\newcommand{\re}{\mathop{\mathrm{Re}}}
\newcommand{\im}{\mathop{\mathrm{Im}}}

\newcommand{\fdet}[2][]{{\mathrm{det}}_{#1}\!\left(#2\right)}

\newcommand{\prob}[2][]{\mathscr P_{#1}\! \ifstrempty{#2}{}{\left\{#2\right\}}}
\newcommand{\myexp}[1]{\exp{\!\left( {#1} \right)}}
\newcommand{\GammaF}[1]{\Gamma\! \left(#1\right)}

\newcommand{\R}{\mathbb{R}}
\newcommand{\C}{\mathbb{C}}
\newcommand{\Rpl}{\mathbb{R}_{>0}}
\newcommand{\Y}{\mathbb{Y}}
\newcommand{\T}{\mathbb{T}}
\newcommand{\A}{\mathbb{A}}
\newcommand{\N}{\mathbb{N}}

\newcommand{\Conf}[1]{\mathrm{Conf}\left(#1\right)}
\newcommand{\RSK}{\mathrm{RSK}}
\newcommand{\df}{\overset{\mathrm{def}}{=}}

\newcommand{\Hardedge}{\mathscr{H}}

\newcommand{\Trunc}{\mathscr{T}}
\newcommand{\TruncLog}{\mathscr{T}^{\mathrm{log}}}
\newcommand{\Crit}{\mathfrak{C}}
\newcommand{\Time}{\mathfrak{T}}

\newcommand{\Gf}[5]{G^{#1}_{#2}\!\left(
    \begin{matrix}
      #3\\
      #4
    \end{matrix}
    \biggl| #5\right)
}
\newcommand{\Type}[1]{\mathrm{Type}\left(#1\right)}
\newcommand{\Row}[1]{\mathrm{Row}\left(#1\right)}
\newcommand{\Col}[1]{\mathrm{Col}\left(#1\right)}
\newcommand{\Shape}[1]{\mathrm{Sh}\left(#1\right)}
\newcommand{\Schur}{\mathcal{S}}
\newcommand{\Tabl}[1]{\mathcal{#1}}
\newcommand{\Path}[1]{\mathrm{#1}}

\newcommand{\Geom}{\mathbf{Geom}}
\newcommand{\Exp}{\mathbf{Exp}}

\newcommand{\limntoinf}{\underset{n \to \infty}{\longrightarrow}}
\newcommand{\limnutoinf}{\underset{\nu \to \infty}{\longrightarrow}}
\newcommand{\limalptoinf}{\underset{\alpha \to \infty}{\longrightarrow}}

\newcommand{\limalptoinffd}{\underset{\alpha \to \infty}{\overset{\mathrm{fd}}{\longrightarrow}}}

\renewcommand{\kappa}{\varkappa}
\renewcommand{\phi}{\varphi}
\renewcommand{\epsilon}{\varepsilon}

\begin{document}
\title{Last-passage percolation and product-matrix ensembles}
\author{Sergey Berezin* \orcidlink{0000-0003-4885-7794}}
\address{Department of Mathematics, KTH Royal Institute of Technology, Lindstedtsv\"{a}gen 25, 10044~Stockholm, Sweden; Department of Mathematics, Katholieke Universiteit Leuven, Celestijnenlaan 200 B, bus 2400, Leuven~B-3001, Belgium; St. Petersburg Department of V.A. Steklov Mathematical Institute of RAS, Fontanka 27, St. Petersburg, 191023, Russia}
\email{(Corr.) s\_berezin@rhmath.xyz}
\author{Eugene Strahov \orcidlink{0000-0001-5687-9753}}
\address{Department of Mathematics, The Hebrew University of Jerusalem, Givat Ram, Jerusalem~91904, Israel}
\email{strahov@math.huji.ac.il}
\begin{abstract}
	We introduce and study a model of directed last-passage percolation in planar layered environment. This environment is represented by an array of random exponential clocks arranged in blocks, for each block the average waiting times depend only on the local coordinates within the block. The last-passage time, the total time needed to travel from the source to the sink located in a given block, maximized over all the admissible paths, becomes a stochastic process indexed by the number of blocks in the array. We show that this model is integrable, particularly the probability law of the last-passage time process can be determined via a Fredholm determinant of the kernel that also appears in the study of products of random matrices. Further, we identify the scaling limit of the last-passage time process, as the sizes of the blocks become infinitely large and the average waiting times become infinitely small. Finite-dimensional convergence to the continuous-time critical stochastic process of random matrix theory is established.
\end{abstract}

\maketitle
\tableofcontents

\section{Introduction}
Within the last several decades, random matrix theory has become a universal tool for explaining a broad and ever-growing range of phenomena in combinatorics, statistical mechanics, and other areas of research. The Airy stochastic process, introduced by Pr\"{a}hofer and Spohn~\cite{PrahoferSpohn}, serves as a canonical example of such universality. While this process can be thought of as the top curve in a particular scaling limit of the Dyson Brownian motion---which is ultimately related to the Gaussian unitary ensemble---its importance extends far beyond. Indeed, the same process also emerges in the asymptotic study of seemingly unrelated models such as tandem queues in queuing theory, last-passage percolation and polynuclear growth in integrable probability, and others. We refer the reader to Quastel and Remenik~\cite{QuastelRemenik} for a review about the Airy stochastic process, to Johansson~\cite{JohansssonDiscretePolynuclearGrowth} for various results on convergence of the discrete polynuclear growth processes to the Airy stochastic process, and to the book by Baik, Deift, and Suidan~\cite{BaikDeiftSuidanBook}, where several problems in combinatorics and statistical mechanics are analyzed with the tools of random matrix theory.

One of the  recent achievements in studying random matrices is the discovery that the product-matrix ensembles are integrable, see Akemann and Burda~\cite{Akemann1}, Akemann, Kieburg, and Wei~\cite{AkemannKieburgWei}, and Akemann, Ipsen, and Kieburg~\cite{AkemannIpsenKieburg}. For example, recall that a Ginibre matrix has standard i.i.d. complex Gaussian variables as its entries, and consider a sequence of independent copies of such matrices~$\left\{G_k \right\}_{k \in \N}$, the size of~$G_k$ being~$(n+\nu_k-1)\times(n+\nu_{k-1}-1)$, where $\nu_0=1$ and~$\nu_k \geq 1$ for~$k \in \N$. Then, the squared singular values of the partial product~$Y_k = G_k\times\cdots\times G_1$ form a determinantal point process on~$\Rpl$. In particular, if~$k=1$, one finds the classical Laguerre unitary ensemble. The discovery of integrability motivated one of the present paper's authors to introduce a class of processes called \textit{(multi-time) product-matrix point processes}, see Strahov~\cite{StrahovD}. To give an example of such a process, we denote by~$y_j^{(k)}$, $j=1,\ldots,n$, the squared singular values of~$Y_k$ defined above. Then, the random set of all pairs~$(k,\lambda^{(k)}_j)$, where~ $j=1, \ldots,n$ and $k \in \N$, forms a determinantal point process on~$\N \times \Rpl$ called the \textit{Ginibre product-matrix point process}. We note that product-matrix point processes on~$\N \times \Rpl$ can be viewed as time-dependent or \textit{dynamical} point processes on~$\Rpl$, with time represented by natural numbers.

Borodin, Gorin, and Strahov~\cite{BorodinGorinStrahov} found that certain product-matrix point processes can be understood as scaling limits of Schur (and even of more general) point processes of combinatorics. In particular, it was shown that the product-matrix point process related to truncated Haar-distributed unitary matrices appears naturally in the study of the Young diagrams associated to random skew plane partitions. We will call this process the \textit{truncated-unitary product-matrix point process}. This finding introduced a new angle in understanding why processes originating outside random matrix theory manifest as scaling limits of random matrix models. In the present paper, we will leverage the very same combinatorial connection to study a novel model of directed last-passage percolation, which we named the \textit{directed last-passage percolation in layered environment}.

A general simple directed last-passage percolation model can be described as follows. Consider a lattice featuring two distinguished points, the source and the sink, and a set of independent random clocks, each clock attached to its own vertex. A particle is placed at the source and moves in a prescribed direction towards the sink, being halted at each vertex until the attached random clock expires. The total (random) time to travel from the source to the sink, maximized over all the admissible paths, is called the \textit{last-passage (percolation) time}. We note in passing that the first-passage percolation model, while similar in name, instead operates with the minimum over all the admissible paths and requires very different tools for its analysis.

In our particular version of the general model, we describe the layered environment as an array~$A^{(k)}$ of size~$n \times L_k$ comprised of blocks~$\left(B^{(1)},\ldots, B^{(k)}\right)$ of independent exponential random variables such that each entry~$B_{i,j}^{(\ell)}$ has intensity~$\nu_\ell+i+j-2$, where~$\nu_\ell \ge 1$ and~$\ell=1, \ldots, k$, and represents random clocks attached to the vertex~$(i,j)$ of the lattice~$\mathbb{Z} \times \mathbb{Z}$. The source and the sink are~$(1,1)$ and~$(n, L_k)$, respectively, and the particle is only allowed to make steps increasing only one of its coordinates and exactly by one. Such steps are often depicted as down/right, having in mind that the y-axis of the lattice is flipped to follow the standard indexing of rows and columns in arrays. Note that within each block~$B^{(\ell)}$ the parameter~$\nu_\ell$ is constant, which corresponds to a particular layer of the environment and gives the name to the model. Since the sink is placed in the $k$th block, the last-passage time~$\Time(k)$ becomes a discrete-time stochastic process~$\left(\Time(k), k \in \N\right)$ called the \textit{last-passage time process}.

The first main results of the present paper is Theorem~\ref{thm21}, which provides a formula for the distribution of the last-passage time process~$\Time(k)$ in terms of a Fredholm determinant with respect to the kernel of the truncated-unitary product-matrix point process.

Our second main result is Theorem~\ref{thm22}, which establishes convergence of the appropriately rescaled last-passage time process to what we call the (continuous-time) \textit{critical stochastic process}. Finite-dimensional distributions of this process are governed by a Fredholm determinant with respect to the \textit{extended critical kernel}, which can be understood as a multi-time analogue of the critical kernel defined in Liu, Wang, and Wang~\cite[Equation (1.15)]{LiuWangWang}. The critical kernel arises in the study of the soft edge of the singular spectrum of a product of i.i.d Ginibre matrices under the assumption that the size of the matrices and the number of factors both grow to infinity in such a way that their ratio is a finite strictly-positive constant. Recalling that the Airy stochastic process is connected to the \textit{top curve} of the Dyson Brownian motion and that the last-passage time is related to the \textit{maximum} over all the admissible paths, we can view the critical stochastic process as an analog of the Airy stochastic process. We note that the critical stochastic process and the extended critical kernel are also related to the Brownian motion on~$GL(N,\C)$, as shown in Ahn~\cite{Ahn1, Ahn2}. Additionally, it is worth mentioning that Borodin and P\'{e}ch\'{e}~\cite{BorodinPeche} relate the Airy stochastic process with a scaling limit of a directed percolation model in a quadrant. The book by Baik, Deift, and Suidan~\cite[Section~10.3]{BaikDeiftSuidanBook} describes special cases of the directed last-passage percolation model in terms of the ensemble related to Wishart random matrices and in terms of the Meixner ensemble. More recent papers on integrable models of last-passage percolation are devoted to the two-time distribution in directed last-passage percolation with geometric weights, see Johansson~\cite{JohanssonTwoTimes}, and to the multi-time joint distribution, see Johansson and Rahman~\cite{JohanssonRahman} and references therein. Despite of the fact that last-passage percolation models are fundamental examples of integrable models, to the best of our knowledge, the relation with product-matrix processes has not been observed before.

Our third main result is a series of theorems, Theorem~\ref{thm23}, Theorem~\ref{thm24}, and Theorem~\ref{thm25}. The first theorem is a multi-time generalization of a theorem in Kuijlaars, Zhang~\cite[Theorem~5.3]{KuijlaarsZhang}. We establish convergence of the extended kernel corresponding to the Ginibre product-matrix point process towards the \textit{extended hard-edge kernel}, which generalizes that in~\cite{KuijlaarsZhang}. The second theorem shows that the same extended kernel can be obtained from the extended kernel of the truncated-unitary product-matrix point process. The last theorem shows that the extended critical kernel can be viewed as a scaling limit of the extended hard-edge kernel. This result can be interpreted as a hard-to-soft edge transition. A similar effect is well known for the Bessel kernel as its parameter is sent to infinity, see Borodin and Forrester~\cite[Section~4]{BorodinForrester}. This completes the circle by tying together the percolation in the layered environment and the product-matrix point processes. It is worth pointing out that for a particular set of parameters, the hard-edge (extended) kernel becomes the celebrated Bessel kernel, which governs the Bessel point process. This kernel is also intimately related to directed last-passage percolation, as was first observed by Forrester, and is discussed in Johansson~\cite{JohanssonLastPassagePaper}.

The rest of the paper is organized as follows. In \textbf{Section~\ref{sect_2}}, we introduce the necessary basics about directed last-passage percolation and state two of our main results, Theorem~\ref{thm21} and Theorem~\ref{thm22}, in Subsection~\ref{sect_21} and Subsection~\ref{Section2.2}, respectively. Further, in Paragraph~\ref{sect_231} and Paragraph~\ref{subsect:trunc}, we introduce basic matrix-product point processes, the Ginibre and truncated-unitary product-matrix point process, respectively. We formulate two more of our results, Theorem~\ref{thm23} and Theorem~\ref{thm24} for the hard-edge scaling limit. In Paragraph~\ref{subsubsect:crit_trans}, we give the statement of our last main result, Theorem~\ref{thm25} about a hard-to-soft edge transition to the extended critical kernel. In \textbf{Section~\ref{sect_3}}, we provide background information on the general last-passage percolation problem and introduce last-passage percolation in layered environment. Subsection~\ref{subsect3.1} covers combinatorial (non-stochastic) aspects of last-passage percolation, Subsection~\ref{subsect3.2} focuses on geometrically distributed random clocks, and Subsection~\ref{subsect3.3} on exponentially distributed random clocks. One of the main theorems, Theorem~\ref{thm21}, is proven in Subsection~\ref{sect_34}. In \textbf{Section~\ref{sect_4}}, we prove Theorems~\ref{thm22}--\ref{thm25}. The essential asymptotic tools, such as Lemmas~\ref{lemma44} -- \ref{lemma49} and Proposition~\ref{PropositionMeijFirth}, are introduced in Subsection~\ref{sect_41}. The proofs of Theorem~\ref{thm22} -- \ref{thm25} are given in Subsections~\ref{sect_42} -- \ref{sect_45}, respectively.

\section{Main results}
\label{sect_2}
\subsection{Directed last passage percolation in layered environment}
\label{sect_21}
For~$k \in \N$, introduce an array of random clocks
\begin{equation}
  \label{eq:Barrayblk}
  B^{(k)}=\left(B_{i,j}^{(k)}\right)_{\substack{i=1,\ldots, n\\j = 1,\ldots,\ell_k}}
\end{equation}
of size~$n \times \ell_k$, where~$n, \ell_1, \ell_2, \ldots \in \mathbb{N}$ are natural numbers such that~$\ell_1 \ge n$. Take the entries of the~$B^{(k)}$ to be independent random variables distributed according to the exponential law
\begin{equation}
  \label{eq:exp_laws}
  B_{i,j}^{(k)} \sim \Exp(\nu_k+i+j-2),
\end{equation}
that is,
\begin{equation}
  \prob{B_{i,j}^{(k)}\le x}=1-e^{-\left(\nu_k+i+j-2\right)x},\quad x>0,
\end{equation}
with~$\nu_k \in \N$. Consider a sequence of random arrays
\begin{equation}
  \label{eq:Aarray}
  \renewcommand*{\arraystretch}{1.5}
  A^{(k)}=
  \begin{pmatrix}
    B^{(1)} &\dots &B^{(k)}
  \end{pmatrix},
  \quad k \in \N,
\end{equation}
made up of blocks~$B^{(1)}, B^{(2)}, \ldots$, which are assumed to be independent. Clearly, $A^{(k)}$ is of size~$n \times L_k$, where~$L_k=\ell_1+\ldots+\ell_k$. Note that, even though~$A^{(k)}$ and~$B^{(k)}$ may alternatively be called random matrices, we reserve this term for use in a different context.

We are going to treat the~$A^{(k)}_{i,j}$ as the weights associated with points of~$\mathbb{Z} \times \mathbb{Z}$. Let~$\Pi_{n, k}$ be the collection of all directed paths from~$(1,1)$ to~$(n,L_k)$, going down and to the right (South/East or down/right paths). As it is customary, we flip the y-axis to be consistent with the usual indexing of arrays. Each path~$\Path{p} \in \Pi_{n, k}$ is a sequence,
\begin{equation}
  \mathrm{p}=\left\{\left(i_s,j_s\right)\right\}_{s=1}^{L_k+n-2},
\end{equation}
where~$(i_1,j_1)=(1,1)$, $(i_{L_k+n-2}, j_{L_k+n-2})=(n,L_k)$, and $\left(i_{s+1},j_{s+1}\right)-\left(i_s,j_s\right) \in \{(0,1),(1,0)\}$. An illustration is given in Fig.~\ref{FigureArray}. 
\begin{figure}[ht!]
  \centering
  \begin{tikzpicture}[scale=1.7]
    \begin{scope}
      \draw [->] (-1.5,1.3) -- (-1,1.3) node [yshift=7, xshift=-3]{$j$};
      \draw [->] (-1.5,1.3) -- (-1.5,0.9) node [yshift=7, xshift=-4]{$i$};
      \foreach \y in {-1,-0.5,...,1}{
        \foreach \x in {-1,-0.5,...,2.5}{
          \node[draw,circle,inner sep=1pt,fill] at (\x,\y) {};
        }
        \draw [densely dotted] (2.7,\y) -- (3.3,\y);

	\foreach \x in {3.5,4,...,5}{
          \node[draw,circle,inner sep=1pt,fill] at (\x,\y) {};
        }
      }

      \draw [decorate,decoration={brace,amplitude=8}] (-1.1,1.1) -- (1.1,1.1) node [midway, yshift=15] {$\ell_1$};
      \draw [decorate,decoration={brace,amplitude=8}] (1.4,1.1) -- (2.6,1.1) node [midway, yshift=15] {$\ell_2$};
      \draw [decorate,decoration={brace,amplitude=8}] (3.4,1.1) -- (5.1,1.1) node [midway, yshift=15] {$\ell_k$};
      \draw [decorate,decoration={brace,amplitude=8}] (-1.15,-1.1) -- (-1.15,1.1) node [midway, xshift=-15] {$n$};
      \node at (0.12,0.25) {$B^{(1)}$};
      \node at (2.12,0.25) {$B^{(2)}$};
      \node at (4.4,0.25) {$B^{(k)}$};
      \draw [loosely dotted] (-1.07,-1.07) rectangle (1.07,1.07);
      \draw [loosely dotted] (1.43,-1.07) rectangle (2.57,1.07);
      \draw [loosely dotted] (3.43,-1.07) rectangle (5.07,1.07);
    \end{scope}
    %%%%% Path%%%%%
    \begin{scope}[compass style/.style={color=black}, color=black, decoration={markings,mark= at position 0.6 with {\arrow{stealth}}}]
      \draw[postaction=decorate] (-1,1) -- (-0.5,1);
      \draw[postaction=decorate] (-0.5,1) -- (-0.5,0.5);
      \draw[postaction=decorate] (-0.5,0.5) -- (0,0.5);
      \draw[postaction=decorate] (0,0.5) -- (0.5,0.5);
      \draw[postaction=decorate] (0.5,0.5) -- (0.5,0);
      \draw[postaction=decorate] (0.5,0) -- (1,0);
      \draw[postaction=decorate] (1,0) -- (1.5,0);
      \draw[postaction=decorate] (1.5,0) -- (2,0);
      \draw[postaction=decorate] (2,0) -- (2,-0.5);
      \draw[postaction=decorate] (2,-0.5) -- (2.5,-0.5);
      \draw (2.5,-0.5) -- (2.61,-0.5);
      \draw (3.39,-0.5) -- (3.5,-0.5);
      \draw[postaction=decorate] (3.5,-0.5) -- (4,-0.5);
      \draw[postaction=decorate] (4,-0.5) -- (4,-1);
      \draw[postaction=decorate] (4,-1) -- (4.5,-1);
      \draw[postaction=decorate] (4.5,-1) -- (5,-1);
    \end{scope}
  \end{tikzpicture}
  \caption{An illustration of directed last-passage percolation in layered environment.}
  \label{FigureArray}
\end{figure}
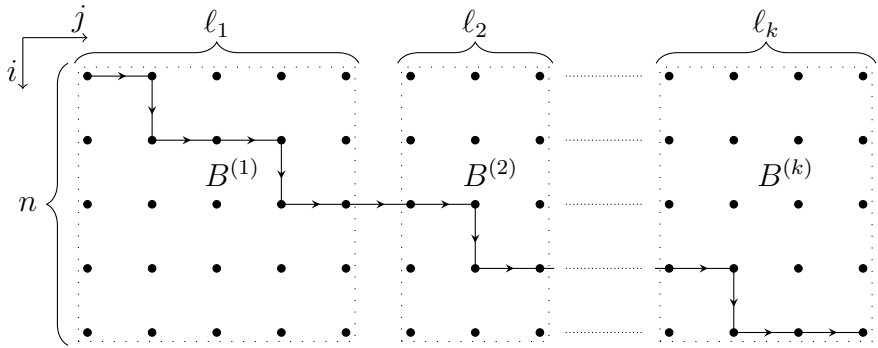
The \textit{last-passage time}~$\Time(k)$ corresponding to the random array~$A^{(k)}$ is defined by
\begin{equation}
  \label{lpp_def}
  \Time(k) \df \max\limits_{\mathrm{p}\in\Pi_{n, k}} \sum\limits_{(i,j)\in\mathrm{p}}A^{(k)}_{i,j},\quad k \in \N.
\end{equation}

One can interpret the quantity~\eqref{lpp_def} as the time, maximized over all admissible paths, that it takes for a particle following South/East paths to reach~$(n,L_k)$ from~$(1,1)$, provided that after jumping to the site~$(i,j)$ this particle experiences a delay of~$A^{(k)}_{i,j}$. The intensity parameter~$\nu_m$ is constant within each block~$B^{(m)}$, $m=1,\ldots, k$, thus we have a layered type of environment, layers corresponding to the blocks~$B^{(m)}$, see Fig.~\ref{FigureArray}.

Our first result gives an explicit formula for the distribution of the \textit{last-passage time process}~$(\Time(k),\ k \in \N)$. Recall that the Meijer $G$-function is defined by
\begin{equation}
  \label{MeijerGFunction}
  \Gf{m,n}{p,q}{a_1, \ldots, a_p}{b_1, \ldots, b_q}{z}=\frac{1}{2\pi i} \int\limits_{\gamma}\frac{\prod\limits_{j=1}^m\GammaF{b_j+\eta}\prod\limits_{j=1}^n\GammaF{1-a_j-\eta}}{\prod\limits_{j=m+1}^q\GammaF{1-b_j-\eta}\prod\limits_{j=n+1}^p\GammaF{a_j+\eta}} z^{-\eta}\, d\eta;
\end{equation}
the choice of the contour~$\gamma$ depends on the parameters~$m,n,p,q$ and can be made in several ways. In particular, all poles of~$\GammaF{b_j+s}$, $j=1,\ldots,m$, should be to the left of~$\gamma$ while those of~$\GammaF{1-a_k-s}$, $k=1,\ldots,n$, to its right. At the same time, an arbitrary contour that separates the poles in the indicated way will not do since the behavior at infinity might not be compatible with the convergence of the integral. A detailed discussion about standard ways of choosing~$\gamma$ is given in Luke~\cite{Luke}. We will not reproduce it here due to its technical nature. Besides, we do not need it: whenever we encounter an explicit integral representation for a Maijer $G$-function, we indicate what particular contour is concerned. 

We have the following theorem.

\begin{thm}
  \label{thm21}
  Let~$r_1, \ldots, r_q \in \N$ be pairwise distinct. Then, the finite-dimensional distributions of the stochastic process~$\left(\Time(k),\ k \in \N \right)$ associated with the arrays~\eqref{eq:Aarray} of the exponential random variables~\eqref{eq:exp_laws} are given by
\begin{equation}
  \label{lpp_findimkern_fredh}
  \prob{\Time(r_k)\leq s_{k},\ k=1,\ldots,N} = \fdet{I-\chi_f K_{n,\vec{\nu},\vec{\ell}} \chi_f}_{L_2(\{r_1,\ldots,r_N\}\times\Rpl)}, \quad N \in \N,
\end{equation}
where~$s_k \in \Rpl$ and the kernel is given by
  \begin{equation}
    \label{ker_trunc_prod_log}
    \begin{aligned}
	    &K_{n,\vec{\nu},\vec{\ell}} (q,x;r,y) =-\Gf{r-q,0}{r-q, r-q}{\nu_{q+1}+\ell_{q+1}-1,  \ldots, & \nu_{r}+\ell_{r}-1}{\nu_{q+1}-1\hfill, \ldots, & \nu_{r}-1\hfill}{e^{x-y}} \mathbf{1}_{r>q}\\
	    &+\int\limits_{S^h_{\sigma}} \frac{d\sigma}{2\pi i}\int\limits_{S_{\zeta}^{(n)}} \frac{d\zeta}{2 \pi i}\, \frac{\prod\limits_{j=0}^{r}\GammaF{\nu_j+\sigma}}{\prod\limits_{j=0}^{q}\GammaF{\nu_j+\zeta}} \frac{\prod\limits_{j=0}^{q}\GammaF{\nu_j+\ell_j+\zeta}}{\prod\limits_{j=0}^{r}\GammaF{\nu_j+\ell_j+\sigma}} \frac{e^{- x \zeta + y\sigma }}{\sigma-\zeta},\ \, x,y \in (0,+\infty),
    \end{aligned}
  \end{equation}
  where by definition~$\nu_0=1$ and $\ell_0 = -n$, the contours of integration are specified in Fig.~\ref{contour2}, and~$\chi_f$ is a multiplication by~$f(q,x)$ defined as follows,
\begin{equation}
  \label{eq:thm1_gfunc}
  f(q,x)= \sum \limits_{j=1}^N \delta_{q,r_j} \mathbf{1}_{(s_j,+\infty)}(x) 
\end{equation}
with~$\delta_{q, r_j}$ being the Kronecker delta.
\begin{figure}[ht!]
	\centering
	\begin{tikzpicture}[scale=2.5]
		\begin{scope}[compass style/.style={color=black}, color=black, decoration={markings,mark= at position 0.5 with {\arrow{stealth}}}]

	% Axes
	% x-axis 
			\draw (-1.9,0) -- (0.625,0);
			\draw[->] (0.875,0) -- (1.875,0);

			\draw (0.765,0) node {$\cdots$};

		% crosses
			\draw (1.5,0) node [cross] {};
			\draw (1,0) node [cross] {};
			\draw (0.5,0) node [cross] {};
			\draw (0,0) node [cross] {};
			\draw (-0.5,0) node [cross] {};
			\draw (-1,0) node [cross] {};
			\draw (-1.5,0) node [cross] {};

		% nodes
			\node at (1.5,-0.075) {\scalebox{0.5}{$n$}};
			\node at (1.0,-0.075) {\scalebox{0.5}{$n-1$}};
			\node at (0.5,-0.075) {\scalebox{0.5}{$1$}};
			\node at (0.05,-0.075) {\scalebox{0.5}{$0$}};
			\node at (-0.535,-0.075) {\scalebox{0.5}{$-1$}};
			\node at (-1.035,-0.075) {\scalebox{0.5}{$-2$}};
			\node at (-1.535,-0.075) {\scalebox{0.5}{$-3$}};

	% y-axis
			\draw[->] (0,-0.525) -- (0,0.525);
	%\zeta-contour
			\draw[thick, postaction=decorate] (-0.25,0.25) -- (-0.25, -0.25);
			\draw[thick, postaction=decorate] (-0.25,-0.25) -- (1.25,-0.25);
			\draw[thick, postaction=decorate] (1.25,-0.25) -- (1.25, 0.25);
			\draw[thick, postaction=decorate] (1.25,0.25) -- (-0.25,0.25);

			\node at (0.3,+0.4) {$S_{\zeta}^{(n)}$};
	%\sigma-contour
			\draw[thick, postaction=decorate] (-0.375,-0.25) -- (-0.375,0.25);
			\draw[thick, postaction=decorate] (-0.375,0.25) -- (-1.75,0.25);
			\draw[thick, postaction=decorate] (-1.75,-0.25) -- (-0.375,-0.25);

			\node at (-1,+0.4) {$S^h_{\sigma}$};
		\end{scope}
	\end{tikzpicture}  
	\caption{The contour~$S^h_{\sigma}$ and~$S_{\zeta}^{(n)}$.} 
	\label{contour2}
\end{figure}
\end{thm}

\begin{remark}
  \label{rem:th1rem1}
  More explicitly,  the Fredholm determinant in~\eqref{lpp_findimkern_fredh} can be written as
  \begin{equation}
    \label{eq:th1_fdredh}
    \begin{aligned}
      &\prob{\Time(r_k)\leq s_{r_k},\ k=1,\ldots,N} \\
      &=1+\sum \limits_{m=1}^\infty \frac{(-1)^m}{m!}\sum\limits_{\substack{\theta_j \in \{r_1,\ldots,r_N\}\\j=1, \ldots, m}} \int \limits_{s_{\theta_1}}^{+\infty} dx_1\cdots \int \limits_{s_{\theta_m}}^{+\infty} dx_m\, \det{\Big(K_{n,\vec{\nu},\vec{\ell}}(\theta_u, x_u;\theta_v,x_v)\Big)}_{u,v=1}^m,
    \end{aligned}
  \end{equation}
  where for convenience we changed~$s_k$ with~$s_{r_k}$. The expression reduces to a finite sum because the corresponding determinantal point process almost surely has a finite number of particles, that is to say, the kernel is that of a finite-rank operator. 
\end{remark}
\begin{remark}
	It is a remarkable fact that since the left-hand side vanishes if at least one of the~$s_k$ is zero, so does the expression on the right-hand side, which is not easy to see directly.
\end{remark}

The kernel~\eqref{ker_trunc_prod_log} is a transformed version of the kernel that describes squared singular values of a product of truncated unitary matrices, as it will be explained at the end of Section~\ref{subsect3.3}. Therefore, Theorem~\ref{thm21} links the last-passage time process~$(\mathfrak{T}(k), k\in\mathbb{N})$ with truncated-unitary product-matrix point processes studied in Borodin, Gorin, and Strahov~\cite{BorodinGorinStrahov}.

\subsection{Convergence to critical stochastic process}
\label{Section2.2}
The goal of this section is to define the critical stochastic process and to establish finite-dimensional convergence of the scaled process~$\left(\Time(k), k\in\N\right)$ to this critical stochastic process.

By definition, the \textit{critical stochastic process}~$\left(\Crit(t), t>0\right)$ is the unique a.s. continuous stochastic process whose finite-dimensional distributions are given by 
\begin{equation}
	\prob{\Crit(t_k)\leq s_k,\; k=1,\ldots,N} =\fdet{I-\chi_f K_\Crit \chi_f}_{L_2\left(\left\{t_1,\ldots,t_N\right\}\times\R\right)},
\end{equation}
where the~$t_k$ are pairwise distinct, $s_k \in \R$, the operator~$\chi_{f}$ is a multiplication by $f(q,x)$ defined in~\eqref{eq:thm1_gfunc} with~$r_j$ replaced with~$t_j$, and~$K_\Crit$ is the \textit{extended critical kernel}
	\begin{equation}
		\label{eq:lim_kern}
		\begin{aligned}
			K_\Crit(\tau,x; t, y) = &-\frac{\mathbf{1}_{t>\tau}}{\sqrt{2 \pi (t-\tau)}} e^{-\frac{1}{2}\frac{(x-y)^2}{t-\tau}}\\
			&+ \int \limits_{{S}_{\sigma}}\frac{d\sigma}{2\pi i} \int \limits_{{S}_{\zeta}} \frac{d\zeta}{2 \pi i} \  \frac{e^{\frac{t\sigma^2}{2}}\GammaF{-\zeta}}{e^{\frac{\tau \zeta^2}{2}}\GammaF{-\sigma}} \frac{e^{-x\zeta+y\sigma}}{\sigma-\zeta}, \quad x,y \in \R,
		\end{aligned}
	\end{equation}
	with the contours of integration shown in Fig.~\ref{contour3}. Note that unlike in the previous section, the~$t_k$ are positive real numbers and not necessarily integers, and~$x$ and~$y$ can be both positive or negative.

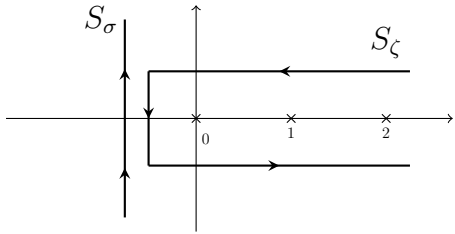
\begin{figure}[ht!]
	\centering
	\begin{tikzpicture}[scale=2.5]
		\begin{scope}[compass style/.style={color=black}, color=black, decoration={markings,mark= at position 0.5 with {\arrow{stealth}}}]
			\draw[->] (-1.,0) -- (1.35,0);
			\draw[->] (0,-0.6) -- (0,0.6);

			\draw[thick, postaction=decorate] (-0.25,-0.25) -- (1.125,-0.25);
			\draw[thick, postaction=decorate] (1.125,0.25) -- (-0.25,0.25);
			\draw[thick, postaction=decorate] (-0.25,0.25) -- (-0.25,-0.25);

			\node at (1,+0.4) {${S}_{\zeta}$};

			\draw (0,0) node [cross] {};
			\draw (0.5,0) node [cross] {};
			\draw (1,0) node [cross] {};
			
			\node at (1.0,-0.075) {\scalebox{0.5}{$2$}};			
			\node at (0.5,-0.075) {\scalebox{0.5}{$1$}};
			\node at (0.05,-0.075) {\scalebox{0.5}{$0$}};

	%\sigma-contour
			\draw[thick, postaction=decorate] (-0.375,-0.525) -- (-0.375,0);
			\draw[thick, postaction=decorate] (-0.375,0) -- (-0.375,0.525);

			\node at (-0.5,0.525) {$S_{\sigma}$};
		\end{scope}
	\end{tikzpicture}
	\caption{The integration contours~$S_{\sigma}$ and~${S}_{\zeta}$. The contour~$S_\sigma$ crosses the real axis at~$-c$, where~$c>0$, and does not intersect~$S_\zeta$.}
	\label{contour3}
\end{figure}
The existence and uniqueness of this process follows from the fact that it can be interpreted as the top curve of a scaling limit of the Dyson Brownian motion with drift. For details, see Ahn~\cite[Theorem~1.3 and Theorem~3.2]{Ahn1}. Also, note that up to the sign, the first term in~\eqref{eq:lim_kern} is the transition density of the standard Wiener process. This is a manifestation of the fact that the process~$\Crit(t)$ is related to the Brownian motion on~$\mathrm{GL}(N,\mathbb{C})$. While our kernel looks different from that of~\cite[Formula (3.2)]{Ahn1}, they are gauge equivalent after a simple change of variables, as one can easily check.

Suppose~$n \df n(\alpha)$, $\vec{\nu} \df (\nu(\alpha), \nu(\alpha), \ldots)$, and $\vec{\ell} \df (\ell(\alpha), \ell(\alpha), \ldots)$, where~$n(\alpha)$, $\nu(\alpha)$, and~$\ell(\alpha)$ are sequences of positive integers indexed by~$\alpha \in \N$. To signify the dependence of the last-passage percolation time in~\eqref{lpp_def} on~$n, \vec{\nu}$, and~$\vec{\ell}$, which now all depend on~$\alpha$, we will write~$\Time_\alpha(k)$ instead of $\Time(k)$.

Below we  need the following notion. Given two sequences~$\{u(\alpha)\}_{\alpha \in \N}$ and~$\{v(\alpha)\}_{\alpha \in \N}$ of positive numbers, we will write $u(\alpha) \ll v(\alpha)$ if for sufficiently large~$\alpha$ there exists~$\delta \in (0,1)$ such that~$(u(\alpha))^{1 + \delta} \le v(\alpha)$ (and thus for all smaller~$\delta$ the inequality holds as well). Note that if~$u(\alpha) \to \infty$ as~$\alpha \to \infty$ and~$u(\alpha) \ll v(\alpha)$, then~$u(\alpha) = o(v(\alpha))$.

\begin{thm}
	\label{thm22}
	Let~$\alpha \to \infty$ in such a way that~$n(\alpha), \nu(\alpha), \ell(\alpha) \to \infty$ and~$n(\alpha) \ll \nu(\alpha) \ll \ell(\alpha)$. Introduce a scaled version of the last-passage time process,

\begin{equation}
	\label{eq:thm1_transf}
	\begin{aligned}
	\Crit_\alpha(t) =&\Time_\alpha([t \nu(\alpha)]) - \log{n(\alpha)}\\
	&- [t \nu(\alpha)] \left(\log{\frac{\nu(\alpha) + \ell(\alpha)}{\nu(\alpha)}} + \frac{\ell(\alpha)}{2\nu(\alpha) (\nu(\alpha) + \ell(\alpha))} \right)+\frac{1}{2n(\alpha)},\quad t>0. 
	\end{aligned}
\end{equation}
Then, the convergence of stochastic processes in finite-dimensional distributions takes place,
	\begin{equation}
		(\Crit_\alpha(t), t>0) \limalptoinffd (\Crit(t), t>0), 
	\end{equation}
	that is,
	\begin{equation}
		\label{eq:thm1_fidi}
		\prob{\Crit_\alpha(t_k) \leq s_{k},\ k=1,\ldots,N}  \underset{\alpha \to \infty}{\longrightarrow}\fdet{I-\chi_f K_\Crit \chi_f}_{L_2(\{t_1,\ldots,t_N\}\times\Rpl)},
	\end{equation}
	where~$t_k \in \Rpl$ are pairwise distinct, $s_k \in \R$, and~$f$ is given by
	\begin{equation}
		f(q,x)= \sum \limits_{j=1}^N \delta_{q,t_j} \mathbf{1}_{(s_j,+\infty)}(x) 
	\end{equation}
	with~$\delta_{q,t_j}$ being the Kronecker delta.
\end{thm}

\begin{remark}
	The condition~$n(\alpha) \ll \nu(\alpha) \ll \ell(\alpha)$ can be relaxed, this however will complicate the proof considerably.
\end{remark}

\subsection{Critical stochastic process and hard-edge limits of product-matrix ensembles}
\label{sect_23}

The main purpose of this section is to relate the critical stochastic process~$(\mathfrak{C}(t), t>0)$ with product-matrix ensembles. It turns out that the critical kernel~\eqref{eq:lim_kern} arises in a particular scaling limit of the singular spectrum of a product of i.i.d Ginibre and of truncated unitary matrices. We will present two theorems describing the hard-edge scaling limit in the Ginibre and truncated-unitary case, together with a theorem that shows a hard-to-soft edge transition. Theorem~\ref{thm23} establishes converges of the extended Ginibre product-matrix kernel to the extended hard-edge kernel, a multi-time generalization of the kernel from Theorem~5.3 in Kuijlaars and Zhang~\cite{KuijlaarsZhang}.
Theorem~\ref{thm24} shows that the very same extended hard-edge kernel emerges in the hard-edge scaling limit for the truncated-unitary product-matrix process. Finally, Theorem~\ref{thm25} states that the extended critical kernel~\eqref{eq:lim_kern} is a scaling limit of the extended hard-edge kernel of Theorem~\ref{thm23} and Theorem~\ref{thm24}.

\subsubsection{Product of Ginibre matrices and its hard-edge scaling limit}
\label{sect_231}
We start by recalling the definition of the \textit{Ginibre product-matrix point process} and its kernel. Let $\{G_k\}_{k \in \N}$ be a sequence of independent Ginibre random matrices of size~$(n+\nu_k-1)\times (n+\nu_{k-1}-1)$, where~$\nu_k \in \N$ and~$\nu_0=1$. Once again, the entries of a Ginibre matrix are i.i.d standard complex Gaussian random variables. Consider matrices~$Y_k^*Y_k$ of size~$n \times n$, where~$Y_k = G_k\times \cdots \times G_1$. Their eigenvalues~$\lambda^{(k)}_j$, the squared \textit{singular values} of~$Y_k$, are almost surely distinct and form a set
\begin{equation}
	\{(k,\lambda^{(k)}_j)|\, {j=1, \ldots,n,\, k \in \N}\}.
\end{equation}
This random set induces a determinantal probability measure on configurations~$\Conf{\N \times \Rpl}$, called the \textit{Ginibre product-matrix point process}~$\mathscr{G}_{n, \vec{\nu}}$. It is proven in Strahov~\cite{StrahovD} that the corresponding extended kernel reads
\begin{equation}
  \label{TimeDependentGinibreKernel}
  \begin{aligned}
	  &K_{n,\vec{\nu}}^{\mathscr{G}}(q,x;r,y)=-\frac{\mathbf{1}_{r>q}}{x}\Gf{r-q,0}{0,r-q}{-}{\nu_{q+1}-1, \ldots, \nu_{r}-1}{\frac{y}{x}}\\
    &+\int\limits_{S^h_{\sigma}} \frac{d\sigma}{2\pi i}
    \int\limits_{S_{\zeta}^{(n)}} \frac{d\zeta}{2 \pi i}\ \frac{\prod \limits_{j=0}^{r} \GammaF{\sigma+\nu_j}}{\prod\limits_{j=0}^{q} \GammaF{\zeta+\nu_j}}
    \frac{\GammaF{\zeta-n+1}}{\GammaF{\sigma-n+1}}\frac{x^\zeta y^{-\sigma-1}}{\sigma-\zeta},\quad x,y>0;\ q, r \in \N;
  \end{aligned}
\end{equation}
where the integration contours are shown in Fig.~\ref{contour2}. The double integral in~\eqref{TimeDependentGinibreKernel} is well defined due to the asymptotic behavior of the gamma function that ensures fast convergence. If~$r>1$, the latter also allows for deforming~$S^h_\sigma$ into a vertical straight line~$S_\sigma$ as in Fig.~\ref{contour3}, if necessary.

\begin{thm}
  \label{thm23}
  For every~$q,r \in \N$,
  \begin{equation}
	  \label{eq:thm1_1}
	  \frac{1}{n}  K_{n,\vec{\nu}}^{\mathscr{G}}\left(q,\frac{x}{n};r,\frac{y}{n}\right)\limntoinf  K^{\Hardedge}_{\vec{\nu}}(q,x;r,y)
  \end{equation}
  uniformly for~$x, y$ in compact subsets of~$\Rpl$, where
  \begin{equation}
    \label{eq:Meij_hardedge}
    \begin{aligned}
	    &K^{\Hardedge}_{\vec{\nu}}(q,x;r,y)=-\frac{\mathbf{1}_{r>q}}{x}\Gf{r-q,0}{0,r-q}{-}{\nu_{q+1}-1, \ldots, \nu_{r}-1}{\frac{y}{x}}\\
	    &+\int\limits_{S^h_{\sigma}} \frac{d\sigma}{2 \pi i}
	    \int\limits_{S_{\zeta}} \frac{d\zeta}{2 \pi i}\ \frac{\prod \limits_{j=1}^{r} \GammaF{\sigma+\nu_j}}{\prod\limits_{j=1}^{q} \GammaF{\zeta+\nu_j}} \frac{\GammaF{-\zeta}}{\GammaF{-\sigma}} \frac{x^\zeta y^{-\sigma-1}}{\sigma-\zeta},\quad x,y>0,\ q, r \in \N,
    \end{aligned}
  \end{equation}
  and the contours~$S^h_\sigma, S_\zeta$ are specified in Fig.~\ref{contour4}. 

\end{thm}
The \textit{extended hard-edge kernel}~\eqref{eq:Meij_hardedge} defines a determinantal point process~$\Hardedge_{\vec{\nu}}$ on~$\N \times \Rpl$ that we call the \textit{hard-edge point process}.
The theorem generalizes Theorem~5.3 of Kuijlaars and Zhang~\cite{KuijlaarsZhang} to the multi-time kernels and can be interpreted as the convergence of the corresponding determinantal point processes. It is worth pointing out that in~\cite{KuijlaarsZhang} the contour~$S_\sigma$ is chosen instead of~$S^h_\sigma$. However, we need the contour as in Fig.~\ref{contour4} to ensure convergence for~$r=1$.

\begin{figure}[ht!]
	\centering
	\begin{tikzpicture}[scale=2.5]
		\begin{scope}[compass style/.style={color=black}, color=black, decoration={markings,mark= at position 0.5 with {\arrow{stealth}}}]
			\draw[->] (-1.9,0) -- (1.35,0);
			\draw[->] (0,-0.6) -- (0,0.6);
			\draw[thick, postaction=decorate] (-0.25,-0.25) -- (1.125,-0.25);
			\draw[thick, postaction=decorate] (1.125,0.25) -- (-0.25,0.25);
			\draw[thick, postaction=decorate] (-0.25,0.25) -- (-0.25,-0.25);

			\node at (1,+0.4) {${S}_{\zeta}$};

			\draw (-1.5,0) node [cross] {};
			\draw (-1,0) node [cross] {};
			\draw (-0.5,0) node [cross] {};
			\draw (0,0) node [cross] {};
			\draw (0.5,0) node [cross] {};
			\draw (1,0) node [cross] {};

			\node at (1.0,-0.075) {\scalebox{0.5}{$2$}};
			\node at (0.5,-0.075) {\scalebox{0.5}{$1$}};
			\node at (0.05,-0.1075) {\scalebox{0.5}{$0$}};
			\node at (-0.575,-0.075) {\scalebox{0.5}{$-1$}};
			\node at (-1.075,-0.075) {\scalebox{0.5}{$-2$}};
			\node at (-1.575,-0.075) {\scalebox{0.5}{$-3$}};

	%\sigma-contour
			\draw[thick, postaction=decorate] (-0.375,-0.25) -- (-0.375,0.25);
			\draw[thick, postaction=decorate] (-0.375,0.25) -- (-1.75,0.25);
			\draw[thick, postaction=decorate] (-1.75,-0.25) -- (-0.375,-0.25);

			\node at (-1,+0.4) {$S^h_{\sigma}$};

		\end{scope}
	\end{tikzpicture}  
	\caption{The integration contours~$S^h_\sigma$ and~${S}_{\zeta}$ in~\eqref{eq:Meij_hardedge}.The contour~$S^h_\sigma$ crosses the real axis at~$-c$, where~$c>0$, and does not intersect~$S_\zeta$. All the poles~$\{ -\nu_0, -\nu_0-1, \ldots; -\nu_1, -\nu_1-1, \ldots\}$  in the~$\sigma$-plane lie inside the domain bounded by~$S^h_\sigma$.}
	\label{contour4}
\end{figure}
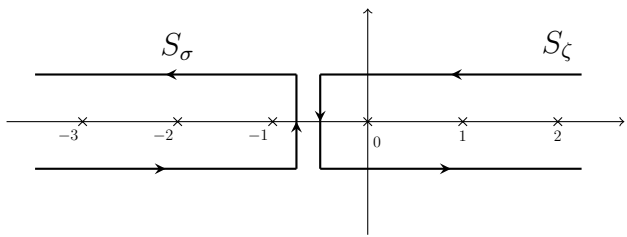

\subsubsection{Product of truncated unitary matrices and its hard-edge scaling limit}
\label{subsect:trunc}
Let~$\{U_k\}_{k \in \mathbb{N}}$ be random independent matrices, each~$U_k$ is of size~$m_k \times m_k$, uniformly distributed over the unitary group~$\mathbf{U}(m_k)$. Denote by~$T_k$ the~$(n+\nu_k-1) \times (n+\nu_{k-1}-1)$-truncation of~$U_k$, 
\begin{equation}
	T_k=
	\begin{pmatrix}
		(U_k)_{1,1} & \cdots & (U_k)_{1,n+\nu_{k-1}-1}\\ 
		\vdots  &  & \vdots \\
		(U_k)_{n+\nu_k-1,1} & \cdots & (U_k)_{n+\nu_k-1,n+\nu_{k-1}-1}
	\end{pmatrix},
\end{equation}
where~$n, \nu_k \in \N$ and we set~$\nu_0=1$. To stay in the realm of absolutely continuous distributions, further we assume that
\begin{equation}
	m_0 \df 0, \quad \ell_0 \df -n,
\end{equation}
\begin{equation}
	\ell_1 \df m_1 -(n+ \nu_1-1) \ge n,
\end{equation}
and
\begin{equation}
	\label{Condition1}
	\ell_k \df m_k-(n+\nu_k-1) \ge 1, \quad k=2,3,\ldots.
\end{equation}
see~\cite[Theorem~4.2.1]{Collins_phd} for more details.

The product~$Y_k = T_k\times \cdots \times T_1$ is a random matrix of size~$(n+\nu_k-1) \times n$. The random set~$\{(k,\widetilde{\lambda}^{(k)}_j)$, $j=1, \ldots,n$, $k \in \mathbb{N}\}$, where the~$\widetilde{\lambda}^{(k)}_j$ are  squared singular values of~$Y_k$ (eigenvalues of~$Y_k^*Y_k$), possesses a determinantal structure and gives rise to a determinantal point process~$\Trunc_{n,\vec{\nu}, \vec{\ell}}$ on~$\mathbb{N} \times (0,1)$. This fact was established by Borodin, Gorin, and Strahov in~\cite{BorodinGorinStrahov}, and explicit formulas were given for the corresponding multi-time kernel~$K^{\Trunc}_{n,\vec{\nu},\vec{\ell}}\,(q,x;r,y)$,
  \begin{equation}
    \label{ker_trunc_prod}
    \begin{aligned}
	    &K_{n,\vec{\nu},\vec{\ell}}^{\Trunc}\,(q,x;r,y)=-\frac{\mathbf{1}_{r>q}}{x}\Gf{r-q,0}{r-q, r-q}{\nu_{q+1}+\ell_{q+1}-1, & \ldots, & \nu_{r}+\ell_{r}-1}{\nu_{q+1}-1\hfill, & \ldots, & \nu_{r}-1\hfill}{\frac{y}{x}} \\
	    &+\int\limits_{S^h_{\sigma}} \frac{d\sigma}{2\pi i}\int\limits_{S_{\zeta}^{(n)}} \frac{d\zeta}{2 \pi i}\, \frac{\prod\limits_{j=0}^{r}\GammaF{\nu_j+\sigma}}{\prod\limits_{j=0}^{q}\GammaF{\nu_j+\zeta}} \frac{\prod\limits_{j=0}^{q}\GammaF{\nu_j+\ell_j+\zeta}}{\prod\limits_{j=0}^{r}\GammaF{\nu_j+\ell_j+\sigma}} \frac{x^{\zeta} y^{-\sigma-1}}{\sigma-\zeta},\ \, x,y \in (0,1),
    \end{aligned}
  \end{equation}
  where the integration contours~$S^h_\sigma$ and~$S_\zeta^{(n)}$ are specified in Fig.~\ref{contour2}. Again, it is worth pointing out that this kernel is a transformed version of the kernel~\eqref{ker_trunc_prod_log}. For more details, see the end of Section~\ref{subsect3.3}.

  In comparison to~\eqref{TimeDependentGinibreKernel}, the kernel~\eqref{ker_trunc_prod} contains an extra factor,
  \begin{equation}
	  \frac{\prod\limits_{j=1}^{q}\GammaF{\nu_j+\ell_j+\zeta}}{\prod\limits_{j=1}^{r}\GammaF{\nu_j+\ell_j+\sigma}},
  \end{equation}
 which will affect our further analysis in a significant way.

\begin{thm}
  \label{thm24}
  Let~$n(\alpha)$ and~$\ell_{j}(\alpha)$, $j \in \mathbb{N}$, be sequences that go to infinity as~$\alpha \to \infty$, in such a way that for each~$j \in \N$ one has~$n(\alpha) \ll \ell_j(\alpha)$. Set
  \begin{equation}
	  f_\alpha(q, x) = \left(q,\frac{x}{n(\alpha) \prod\limits_{j=1}^q (\ell_j(\alpha) + \nu_j(\alpha))}\right).
  \end{equation}

  Then, for every~$q,r \in \N$,
 \begin{equation}
	 \label{eq:thm_hard_eq1}
	 \frac{(n(\alpha))^{-1} \prod\limits_{j=1}^{r} \GammaF{\nu_j + \ell_j(\alpha)}}{\prod\limits_{j=1}^r (\nu_j + \ell_j(\alpha)) \prod\limits_{j=1}^{q} \GammaF{\nu_j + \ell_j(\alpha)}}  {K}^\Trunc_{n(\alpha),\vec{\nu},\vec{\ell}(\alpha)}(f_\alpha(q,x); f_\alpha(r,y)) \limalptoinf K^{\Hardedge}_{\vec{\nu}}(q,x;r,y),
 \end{equation}
 uniformly for~$x,y$ in compact subsets of~$\Rpl$, where~$K^{\Hardedge}_{\vec{\nu}}(q,x;r,y)$ is the kernel~\eqref{eq:Meij_hardedge}.
\end{thm}

Ultimately, since the gauge factor on the left-hand side of~\eqref{eq:thm_hard_eq1} does not change the law of~$\Trunc_{n,\vec{\nu}, \vec{\ell}}$, this theorem states the convergence of the corresponding determinantal point processes.

\subsubsection{Transition to extended critical kernel}
\label{subsubsect:crit_trans}
Our final result illustrates the phenomenon known as the \textit{hard-to-soft edge transition}. In a simple case of the Bessel and Airy point processes, e.g., see Borodin and Forrester~\cite[Section~4]{BorodinForrester}.

Suppose~$\vec{\nu} =(\nu,\nu,\ldots)$, and set
\begin{equation}
	K^{\Hardedge}_{\nu}(q,x;r,y) \df K^{\Hardedge}_{\vec{\nu}}(q,x;r,y).
\end{equation}
We have the following theorem.
\begin{thm}
	\label{thm25}
	Fix~$\tau, t>0$, and introduce a scaled version of the extended hard-edge kernel,
\begin{equation}
	\label{eq:thm25_eq1}
	\widehat{K}_{\nu}(\tau, x; t, y) = \frac{e^{[t \nu]\left(\log{\nu}-\frac{1}{2\nu}\right)-y}}{(\GammaF{\nu})^{[t\nu]-[\tau \nu]}} K_\nu^{\Hardedge}\left([\tau \nu],e^{[\tau \nu]\left(\log{\nu} - \frac{1}{2\nu}\right) - x};[t \nu],e^{[t \nu]\left(\log{\nu} - \frac{1}{2\nu}\right)-y}\right).
\end{equation}
Then, the following convergence takes place
	\begin{equation}
		\label{eq:th5_eq2}
		\widehat{K}_{\nu}(\tau, x; t, y) \limnutoinf {K}_{\Crit}(\tau, x; t, y),
	\end{equation}
	uniformly for~$x$ and~$y$ in compact subsets of~$\Rpl$. The kernel on the right-hand side of~\eqref{eq:th5_eq2} is the extended critical kernel~\eqref{eq:lim_kern}.
\end{thm}

It is worth noticing that this convergence can be interpreted as the convergence of certain line ensembles. For details, we refer the reader to the paper by Corwin and Hammond~\cite{CorwinHammond}.

\section{Last-passage percolation} 
\label{sect_3}
\subsection{Young diagrams, tableaux, and integer arrays}
\label{subsect3.1}
We recall three main objects that we will extensively use further. The first is the set~$\Y$ of \textit{partitions}, which is a collection of all weakly decreasing sequences of non-negative integers,
\begin{equation}
  \label{eq:shp}
  \lambda=(\lambda_1, \lambda_2, \ldots), \quad  \lambda_j \in \N \cup \{0\}, \quad \lambda_1 \ge \lambda_2 \ge \ldots,
\end{equation}
of finite \textit{weight}~$|\lambda|$,
\begin{equation}
  |\lambda| = \sum \limits_{j=1}^\infty \lambda_j.
\end{equation}
The number of non-zero elements of~$\lambda$ is denoted by~$\ell(\lambda)$ and called the \textit{length} of~$\lambda$.

Partitions are in one-to-one correspondence with Young diagrams, where each partition~$\lambda$ is identified with the \textit{Young diagram} via
\begin{equation}
  \label{eq:Youngd}
  \lambda \mapsto \{(i,j)\in \N \times \N| 1 \le j \le \lambda_i,\ i = 1, \ldots, \ell(\lambda) \}.
\end{equation}
 Because of this identification, we are going to use the term partitions and the term Young diagrams interchangeably.

The second object is the set~$\T$ of \textit{semi-standard Young tableaux}. Given a partition~$\lambda$, a corresponding tableau~$\Tabl{T}$ is an array on positive integers,
\begin{equation}
  \Tabl{T} = (\Tabl{T}_{i,j})_{\substack{i=1,\ldots,\ell(\lambda)\\j=1,\ldots,\lambda_i}},\quad \Tabl{T}_{i,j} \in \N,
\end{equation}
such that~$\Tabl{T}_{i,j}$ increases strictly with respect to~$i$ and increases weakly with respect to~$j$. It is said that the \textit{shape} of~$\Tabl{T}$ is~$\lambda$,
\begin{equation}
  \Shape{\Tabl{T}} = \lambda.
\end{equation}

Let~$\lambda$ and~$\mu$ be two partitions such that~$\lambda \succeq \mu$ (i.e., $\lambda_i \ge \mu_i$, $i \in \N$). A \textit{semi-standard skew tableau~$\Tabl{T}$ of shape~$\lambda/\mu$} is an array of positive integers
\begin{equation}
  \Tabl{T} = (\Tabl{T}_{i,j})_{\substack{i=1,\ldots,\ell(\lambda)\\j=\mu_i+1,\ldots,\lambda_i}}, \quad \Tabl{T}_{i,j} \in \N,
\end{equation}
that are called \textit{labels} and that increases strictly in~$i$ and weakly in~$j$. We use the notation
\begin{equation}
  \Shape{\Tabl{T}} = \lambda/\mu.
\end{equation}
Note that any semi-standard tableau is a skew tableau of shape~$\lambda/\varnothing$, where~$\varnothing=(0,0, \ldots)$ is the empty Young diagram.

A tableau~$\Tabl{T}$ of shape~$\lambda/\mu$ can be identified with the sequence of partitions
\begin{equation}
  \label{eq:tabl_as_seq}
  \mu=\upsilon^{(0)} \preceq \upsilon^{(1)} \ldots \preceq \upsilon^{(q)}=\lambda
\end{equation}
such that each skew diagram~$\upsilon^{(i)}/\upsilon^{(i-1)}$ is a horizontal strip, i.e., has at most one square in each column. Every strip~$\upsilon^{(i)}/\upsilon^{(i-1)}$, $i=1,\ldots,q$, is identified with the entries of~$\Tabl{T}$ that contain the number~$i$. This gives an alternative definition of a tableau. For more details, see Macdonald~\cite{Macdonald}.

The \textit{type} of a skew tableau~$\Tabl{T}$ is defined by
\begin{equation}
  \Type{\Tabl{T}} = (\#\{(i,j)|\Tabl{T}_{i,j}=1\}, \#\{(i,j)|\Tabl{T}_{i,j}=2\}, \ldots),
\end{equation}
where~$\#\{(i,j)|\Tabl{T}_{i,j}=a\}$, $a \in \N$, is the number of elements of~$\Tabl{T}$ equal to~$a$. The corresponding \textit{Schur function} is defined as
\begin{equation}
	\label{eq:Schur_funct}
	s_{\lambda/\mu}(x) = \sum\limits_{\Shape{\Tabl{T}} = \lambda/\mu} x^{\Type{\Tabl{T}}},
\end{equation}
where~$x = (x_1, x_2,\ldots)$, the sum is taken over all semi-standard skew Young tableaux~$\Tabl{T}$ of the specified shape, and 
\begin{equation}
x^{\Type{\Tabl{T}}} = (x_1^{\Type{\Tabl{T}}_1}, x_2^{\Type{\Tabl{T}}_2}, \ldots).
\end{equation}
If~$\mu=\varnothing$, we simply write~$s_\lambda(x)$.

The last but not least object is the set~$\A$ of infinite arrays with non-negative integer entries
\begin{equation}
	A=\left(A_{i,j}\right)_{i,j \in \N}, \quad A_{i,j} \in \N \cup \{0\},
\end{equation}
such that the \textit{weight}~$|A|$ is finite,
\begin{equation}
  |A| = \sum_{i,j=1}^\infty A_{i,j} < +\infty.
\end{equation}

The \textit{row type} of~$A$ is defined by
\begin{equation}
  \Row{A}=\left(\sum\limits_{j=1}^\infty A_{1,j}, \sum\limits_{j=1}^\infty A_{2,j}, \ldots \right),
\end{equation}
and the \textit{column type} of~$A$ is defined by
\begin{equation}
  \Col{A}=\left(\sum\limits_{i=1}^\infty A_{i,1}, \sum\limits_{i=1}^\infty A_{i,2}\ldots \right).
\end{equation}

Since finite arrays with non-negative entries are embedded in~$\A$ in a natural way (via extension by zeros), we always think of finite arrays as being elements of~$\A$.

The following theorem is a well-known version of the Robinson--Schensted--Knuth (RSK) correspondence. For more information, see Stanley~\cite{Stanley} and also Baik, Deift, and Suidan~\cite{BaikDeiftSuidanBook}. 
\begin{thm}
  \label{TheoremRSK}
  There exists a bijective map~$\RSK\!\!: \A \to \T \times \T$,
  \begin{equation}
    A \overset{\RSK}{\mapsto} (\Tabl{P}(A),\Tabl{Q}(A)),
  \end{equation}
given constructively by the RSK algorithm, such that
  \begin{equation}
    \Shape{\Tabl{P}(A)} = \Shape{\Tabl{Q}(A)}
  \end{equation}
  and
  \begin{equation}
    \Col{A}=\Type{\Tabl{P}(A)},\quad \Row{A}=\Type{\Tabl{Q}(A)}.
  \end{equation}
\end{thm}

The RSK algorithm produces~$\Tabl{P}(A)$ by carrying out sequential RSK insertions row by row, going along each row from left to right. Each time the~$j$-coordinate of the corresponding element of the array~$A$ is inserted in~$\Tabl{P}(A)$ as many times as prescribed by~$A_{i,j}$,
\begin{equation}
  \label{eq:proof32_eq3}
  \Tabl{P}(A) = (\ldots(\varnothing \gets g_1)\gets g_2\ldots) \gets g_N, \quad g_s\in \N,
\end{equation}
where
\begin{equation}
	N = |A|,
\end{equation}
and the~$i$-coordinate is inserted at the same position but into~$\Tabl{Q}(A)$ instead.

Introduce the operator~$d_{k}(\cdot): \T \to \T$ that erases all the elements of a tableau that are greater than~$k$. By using the alternative definition of a tableau via~\eqref{eq:tabl_as_seq} and by erasing the elements one by one starting from the maximal, we see that for every~$\Tabl{T} \in \T$ and~$k \in \N$, the object~$d_{k}(\Tabl{T})$ is also a semi-standard tableau. 
We are going to need the following simple lemma, which is not obvious due to the ``non-commutative" nature of the RSK insertions. Nevertheless, the proof is quite simple and is given below for the reader's convenience.
\begin{lem}
  \label{lemma:commutativity}
  \begin{equation}
    d_{k}((\ldots(\varnothing \gets j_1)\gets j_2\ldots) \gets j_N) = (\ldots(\varnothing \overset{\le k}{\longleftarrow} j_1) \overset{\le k}{\longleftarrow} j_2\ldots) \overset{\le k}{\longleftarrow} j_N,
  \end{equation}
  where the operation~$\overset{\le k}{\longleftarrow} j$ RSK-inserts~$j$ into the tableau if ~$j \le k$ and does nothing otherwise.
\end{lem}

\begin{proof}
  We prove the lemma by using induction with respect to the number of insertions. The base case,
  \begin{equation}
    d_{k}(\varnothing \gets j_1) = \varnothing \overset{\le k}{\longleftarrow} j_1 
  \end{equation}
  is trivial. It remains to carry out the inductive step. Suppose that
  \begin{equation}
    \label{eq:proof32_indhyp}
    d_{k}(\Tabl{T}_s) = \widetilde{\Tabl{T}}_s,
  \end{equation}
  where
  \begin{equation}
    \Tabl{T}_s \df (\ldots(\varnothing \gets j_1)\gets j_2\ldots) \gets j_s
  \end{equation}
  and
  \begin{equation}
    \widetilde{\Tabl{T}}_s \df (\ldots(\varnothing \overset{\le k}{\longleftarrow} j_1) \overset{\le k}{\longleftarrow} j_2\ldots) \overset{\le k}{\longleftarrow} j_s.
  \end{equation}

  Consider the insertion path~$I(\Tabl{T}_s \gets j_{s+1})$ of the element~$j_{s+1}$ into the tableau~$\Tabl{T}_s$. Namely, this path marks the positions of the elements in~$\Tabl{T}_s$, together with their labels, that have changed as a result of the RSK insertion of~$j_{s+1}$. Note that the labels of the marked elements form a strictly increasing sequence as we go along the path. Define the ``stopped" RSK insertion~$\overset{\phantom{2}\le k}\Longleftarrow$, which only takes into account the elements of~$I(\Tabl{T}_s \gets j_{s+1})$ with the labels less or equal to~$k$. This is to say, the RSK insertion stops right before it encounters the first element strictly greater than~$k$ that otherwise would have been bumped into the next row. In particular, if~$j_{s+1} >k$ then~$\overset{\phantom{2}\le k}\Longleftarrow j_{s+1}$ does nothing. Note that this operation preserve the structure of a semi-standard Young tableaux.

  Clearly,
  \begin{equation}
    \label{eq:proof32_eq5}
    d_{k}(\Tabl{T}_s\gets j_{s+1}) = d_{k}(\Tabl{T}_s \overset{\phantom{2}\le k}\Longleftarrow j_{s+1} ) = (d_{k}(\Tabl{T}_s) \overset{\le k}{\longleftarrow} j_{s+1}),
  \end{equation}
  and by invoking the inductive hypothesis~\eqref{eq:proof32_indhyp}, we arrive at
  \begin{equation}
    d_{k}((\ldots(\varnothing \gets j_1)\gets j_2\ldots) \gets j_{s+1}) = (\ldots(\varnothing \overset{\le k}{\longleftarrow} j_1) \overset{\le k}{\longleftarrow} j_2\ldots) \overset{\le k}{\longleftarrow} j_{s+1}.
  \end{equation}
  This concludes the induction and finishes the proof.
\end{proof}

Before we proceed, note that the erasing operator~$d_k(\cdot)$, defined earlier on~$\T$, can be also defined on the space of arrays~$\A$. For a given array, this operator replaces all the columns after the $k$th with zeros. Having in mind the previous lemma, we can now establish that the RSK operator commutes with the erasing operator.
\begin{prop}
	\label{prop:prop3.3}
	Let~$A^{(k)}$ be given by~\eqref{eq:Barrayblk}--\eqref{eq:Aarray}, and set~$\RSK(A^{(k)}) = (\Tabl{P}(A^{(k)}),\Tabl{Q}(A^{(k)}))$, where
  \begin{equation}
	  \Shape{\Tabl{P}(A^{(k)})} = \Shape{\Tabl{Q}(A^{(k)})}.
  \end{equation}
  Then, 
  \begin{equation}
	  \Tabl{P}(A^{(k)}) = d_{L_k}(\Tabl{P}(A^{(k+1)})).
  \end{equation}
\end{prop}
\begin{proof}
	Fix~$k$ and write~$\Tabl{P}(A^{(k+1)})$ as a sequence of RSK insertions,
  \begin{equation}
	  \Tabl{P}(A^{(k+1)}) = (\ldots(\varnothing \gets g_1)\gets g_2\ldots) \gets g_{|A^{(k+1)}|}.
  \end{equation}
  Then, the proof follows by applying Lemma~\ref{lemma:commutativity} and noticing that
  \begin{equation}
	  \Tabl{P}(A^{(k)}) = (\ldots(\varnothing \overset{\le L_k}{\longleftarrow} g_1) \overset{\le L_k}{\longleftarrow} g_2\ldots) \overset{\le L_k}{\longleftarrow} g_{|A^{(k+1)}|}.
  \end{equation}
\end{proof}

\subsection{Arrays of geometric random variables and Schur point processes}
\label{subsect3.2}
In this section we assume that elements of
\begin{equation}
  \label{eq:Akdef_geomsect}
  A^{(k)} = (B^{(1)}\cdots B^{(k)}) 
\end{equation}
in~\eqref{eq:Aarray} are independent random variables distributed according to the geometric law,
\begin{equation}
  \label{eq:Bgeom}
  B^{(k)}_{i,j} \sim \Geom(1-x_i y_j^{(k)}), 
\end{equation}
that is,
\begin{equation}
  \label{eq:geom_weights}
  \prob{(B^{(k)})_{i,j} = s} = (1 - x_i y_j^{(k)})(x_i y_j^{(k)})^s, \quad s \in \N \cup \{0\},
\end{equation}
where~$x_i, y_j^{(k)} \in (0,1)$, $i = 1, \ldots, n$, and $j = 1, \ldots, \ell_k$. Note that the~$x_i$ are the same for different blocks~$B^{(k)}$. 

Theorem~\ref{TheoremRSK} allows one to define a family of random elements
\begin{equation}
  \label{eq:Lambdaprocess}
  \Lambda^{(k)}= \Shape{\Tabl{P}(A^{(k)})},\quad k \in \N,
\end{equation}
which take on values in the space of Young diagrams~$\Y$ and form a discrete-time stochastic process~$(\Lambda^{(k)},\ k \in \N)$. The next proposition provides explicit formulas for finite-dimensional distributions of this process.
\begin{prop}
  \label{prop:Schur_proc}
  Let elements of~$A^{(k)}$ be independent geometric random variables as in~\eqref{eq:Akdef_geomsect} and~\eqref{eq:Bgeom}. Then, finite dimensional distributions of~$(\Lambda^{(k)},\ k \in N)$ are given by
  \begin{equation}
    \label{eq:Schur_proc}
    \prob{\Lambda^{(k)} = \lambda^{(k)},\ k=1,\ldots,q} = \frac{1}{Z_{\Schur}} s_{\lambda^{(q)}}(x) \prod_{k=1}^q s_{\lambda^{(k)}/\lambda^{(k-1)}} (y^{(k)}),
  \end{equation}
  where~$x = (x_1,\ldots,x_n)$, $y^{(k)}=(y_1^{(k)},\ldots, y_{\ell_k}^{(k)})$, by definition $\lambda^{(0)} = \varnothing$, and~$s_{\lambda^{(k)}/\lambda^{(k-1)}} = 0$ if~$\lambda^{(k-1)} \preceq \lambda^{(k)}$ is not satisfied.  The normalization factor is given by
  \begin{equation}
    \label{eq:norm_fact}
    \frac{1}{Z_{\Schur}} = \prod_{\substack{1 \le i \le n\\1 \le k \le q\\1\le j \le \ell_k}} \left(1-x_iy_j^{(k)}\right).
  \end{equation}
\end{prop}
\begin{remark}
	Clearly, the process~$\left(\Lambda^{(k)},\ k \in \N\right)$ is Markov.
\end{remark}

For a fixed~$k \in \N$, the law of~$\Lambda^{(k)}$  is a measure on~$\Y$ known as the \textit{Schur measure}, see Okounkov~\cite{Okounkov}, Okounkov and Reshetikhin~\cite{OkounkovReshetikhin}, and also Baik, Deift, and Suidan~\cite{BaikDeiftSuidanBook}. Likewise, the process~$(\Lambda^{(k)},\ k \in \N)$ is often called the \textit{Schur stochastic process}. 
\begin{proof}
	If~$\lambda^{(1)} \preceq \ldots \preceq \lambda^{(q)}$ is not satisfied, immediately
  \begin{equation}
	  \prob{\Lambda^{(k)} = \lambda^{(k)},\ k=1,\ldots,q} = 0.
  \end{equation}

  Now, suppose~$\lambda^{(1)} \preceq \ldots \preceq \lambda^{(q)}$ holds, and write
  \begin{equation}
    \begin{aligned}
      P &= \prob{\Lambda^{(k)} = \lambda^{(k)},\ k=1,\ldots,q} = \prob{\Shape{\Tabl{P}(A^{(k)})} = \lambda^{(k)}, k=1,\ldots,q} \\
      &= \sum \limits_{\substack{\Shape{\Tabl{P}(b^{(1)}\ldots b^{(k)})}= \lambda^{(k)}\\k=1,\ldots, q}} \prob{B^{(k)} = b^{(k)},\ k=1,\ldots,q}.
    \end{aligned}
  \end{equation}
  By using the independence of the weights, we see that
  \begin{equation}
    \begin{aligned}
      P &= \sum \limits_{\substack{\Shape{\Tabl{P}(b^{(1)}\ldots b^{(k)})}= \lambda^{(k)}\\k=1,\ldots, q}}  \prod_{\substack{1 \le i \le n\\1 \le k \le q\\1\le j \le \ell_k}} \prob{B^{(k)}_{i,j} = b^{(k)}_{i,j}}\\
      &=\sum \limits_{\substack{\Shape{\Tabl{P}(b^{(1)}\ldots b^{(k)})}= \lambda^{(k)}\\k=1,\ldots, q}}  \prod_{\substack{1 \le i \le n\\1 \le k \le q\\1\le j \le \ell_k}} (1-x_iy_j^{(k)}) (x_iy_j^{(k)})^{b_{i,j}^{(k)}}.
    \end{aligned}
  \end{equation}
  Collecting the factors, we find that
  \begin{equation}
    \begin{aligned}
      &P =\frac{1}{Z_\Schur}\sum \limits_{\substack{\Shape{\Tabl{P}(b^{(1)}\ldots b^{(k)})}= \lambda^{(k)}\\k=1,\ldots, q}}  x^{\sum\limits_{k=1}^q\Row{b^{(k)}}} \prod_{k=1}^q  (y^{(k)})^{\Col{b^{(k)}}}\\
      &=\frac{1}{Z_\Schur}\sum \limits_{\substack{\Shape{\Tabl{P}(a^{(k)})}= \lambda^{(k)}\\k=1,\ldots, q}}  x^{\Row{a^{(q)}}} y^{\Col{a^{(q)}}},
    \end{aligned}
  \end{equation}
where~$y = (y_1, \ldots,y_{L_q}) \df (y^{(1)}_1,\ldots,y^{(1)}_{\ell_1},\ldots, y^{(q)}_1, \ldots,y^{(q)}_{\ell_q}) $.

  Using Proposition~\ref{prop:prop3.3} and Theorem~\ref{TheoremRSK}, we can write
  \begin{equation}
    \begin{aligned}
      P &=\frac{1}{Z_\Schur}\sum \limits_{\substack{\Shape{d_{L_k}\left(\Tabl{P}(a^{(q)})\right)}= \lambda^{(k)}\\k=1,\ldots, q}}  x^{\Type{\Tabl{Q}(a^{(q)})}} y^{\Type{\Tabl{P}(a^{(q)})}}\\
      &= \frac{1}{Z_\Schur}\sum \limits_{\substack{\Shape{d_{L_k}(\Tabl{P})}= \lambda^{(k)}\\ \Shape{\Tabl{Q}}=\lambda^{(q)}\\ k=1,\ldots, q}}  x^{\Type{\Tabl{Q}}} y^{\Type{\Tabl{P}}}.
    \end{aligned}
  \end{equation}
    The latter sum is over all semi-standard tableaux~$\Tabl{P}$ and~$\Tabl{Q}$ of the specified shape. The sum factorizes and, due to~\eqref{eq:Schur_funct}, we arrive at
  \begin{equation}
	  P= \frac{1}{Z_\Schur}s_{\lambda^{(q)}}(x)\sum \limits_{\substack{\Shape{d_{L_k}(\Tabl{P})}= \lambda^{(k)}\\ k=1,\ldots, q}} y^{\Type{\Tabl{P}}}.
  \end{equation}
  Using the alternative definition of the tableaux~$\Tabl{P}$ via~\eqref{eq:tabl_as_seq}, we can write
  \begin{equation}
	  P= \frac{1}{Z_\Schur}s_{\lambda^{(q)}}(x)\sum \limits_{\substack{\lambda^{(k-1)} =\upsilon^{(L_{k-1})}\preceq \ldots \preceq \upsilon^{(L_k)}=\lambda^{(k)}\\k=1,\ldots,q}} y_1^{|\upsilon^{(1)}| - |\upsilon^{(0)}|} \cdot \ldots \cdot y_{L_q}^{|\upsilon^{(L_q)}| - |\upsilon^{(L_{q-1})}|},
  \end{equation}
  where we set~$L_0=0$ and~$\lambda^{(0)} = \varnothing$, and~$\upsilon^{(j)} \in \Y$ are the corresponding Young diagrams such that~$\upsilon^{(j)}/\upsilon^{(j-1)}$ is a horizontal strip.

  Rearranging the sum of products, we find
  \begin{equation}
	  P= \frac{1}{Z_\Schur}s_{\lambda^{(q)}}(x) \prod_{k=1}^q \sum\limits_{\Shape{P}=\lambda^{(k)}/\lambda^{(k-1)}} (y^{(k)})^{\Type{P}} =\frac{1}{Z_\Schur}s_{\lambda^{(q)}}(x) \prod_{k=1}^q s_{\lambda^{(k)}/\lambda^{(k-1)}}(y^{(k)}).
  \end{equation}
 This concludes the proof.
\end{proof}

\subsection{Arrays of exponential random variables and limits of Schur point processes}
\label{subsect3.3}
In this section, we will show an analog of Proposition~\ref{prop:Schur_proc} for the case of exponential weights. In order to do so, we approximate the exponential law by a sequence of geometric laws. Indeed, let~$X_N$ be a random variable distributed geometrically,
\begin{equation}
  X_N \sim \mathrm{Geom}\left(1-e^{-\frac{a}{N}}\right), \quad a>0,\ N \in \N,
\end{equation}
that is,
\begin{equation}
	\prob{X_N= k} = \left(1 - e^{-\frac{a}{N}}\right) e^{-\frac{a k}{N}}, \quad k \in \N \cup \{0\}.
\end{equation}
Then, as it is easy to check, convergence in distribution takes place,
\begin{equation}
  \frac{X_N}{N} \underset{N \to \infty}{\overset{\mathrm{d}}{\longrightarrow}} \Exp(a),
\end{equation}
that is,
\begin{equation}
  \prob{\frac{X_N}{N} \le x} \underset{N \to \infty}{\longrightarrow}\int \limits_{0}^x a e^{-a s}\, ds, \quad x>0.
\end{equation}

We consider the model~\eqref{eq:Akdef_geomsect}--\eqref{eq:geom_weights} and choose the~$x_i$ and~$y_j$ in the following way
\begin{equation}
  \label{eq:geom_approx_exp}
  x_i(N)  = e^{- \frac{i-1}{N}}, \quad y_j^{(k)}(N) = e^{-\frac{\nu_k+j-1}{N}},\quad N \in \N,
\end{equation}
where $\nu_k \in \N$. To emphasize dependence on~$N$ of all objects, we are going to use~$N$ as an extra argument. The formula~\eqref{eq:geom_approx_exp} implies
\begin{equation}
  \frac{B^{(k)}_{i,j}(N)}{N} \overset{\mathrm{d}}{\underset{N \to \infty}{\longrightarrow}} \Exp(\nu_k+i+j-2).
\end{equation}

The following proposition holds.
\begin{prop}
  \label{prop:fidi_lambda_conv}
  Let~$(\Lambda^{(k)}(N),\ k \in \N)$ be an $N$-indexed family of stochastic processes, where $\Lambda^{(k)}(N)= \Shape{\Tabl{P}(A^{(k)}(N))}$ and~$\Tabl{P(\cdot)}$ is defined as in Theorem~\ref{TheoremRSK}. Assume that elements of the arrays~$A^{(k)}(N) = \left(B^{(1)}(N) \ldots B^{(k)}(N)\right)$ are independent geometric random variables,
  \begin{equation}
    B^{(k)}_{i,j}(N) \sim \Geom(1-x_i(N) y^{(k)}_j(N)), \quad i=1, \ldots, n;\ j=1, \ldots, \ell_k;\ k \in \N;
  \end{equation}
  where~$x_i(N)$ and~$y^{(k)}_j(N)$ are defined in~\eqref{eq:geom_approx_exp} for some~$\nu_k \ge 1$.

  Then, the convergence of finite-dimensional distributions takes place
  \begin{equation}
	  \left(\frac{\Lambda^{(k)}(N)}{N},\ k \in \N\right) \overset{\mathrm{fd}}{\underset{N \to \infty}{\longrightarrow}} \left(\widetilde{\Lambda}^{(k)},\ k \in \N\right).
  \end{equation}
  The distribution of $\left(\widetilde{\Lambda}^{(1)},\ldots, \widetilde{\Lambda}^{(q)}\right)$, $q \in \N$, is given by
  \begin{equation}
    \label{eq:prop_Mq}
    \begin{aligned}
      &\mathfrak{M}_{q}\left(d\lambda^{(1)}, \ldots, d\lambda^{(q)}\right) =\frac{1}{Z_{n,q}} \det{\left( e^{-(j-1) \lambda_k^{(q)}}\right)}_{j,k=1}^n\\
      &\times \prod_{s=1}^q \det{\left(\frac{e^{(\nu_s+(n-1)\delta_{s,1})\lambda_k^{(s-1)}}}{e^{(\nu_s+\ell_s-1-(n-k)\delta_{s,1})\lambda_j^{(s)}}}\left(e^{\lambda_j^{(s)}} - e^{\lambda_k^{(s-1)}} \right)_+^{\ell_s-1-(n-1)\delta_{s,1}}\right)}_{j,k=1}^n  d\lambda^{(1)} \cdots d\lambda^{(q)}
    \end{aligned}
  \end{equation}
where 
  \begin{equation}
    \label{eq:propZnq}
    {Z}_{n,q} = \prod_{j=1}^n \GammaF{j} \times \prod_{\substack{1 \le j \le n\\1\le k \le q}} \frac{\GammaF{\nu_k+j-1} \GammaF{\ell_k-(n-j)\delta_{k,1}}} {\GammaF{\nu_k+\ell_k+j-1}},
  \end{equation}
  $x_+=\max\{x,0\}$, $\lambda^{(k)} = (\lambda^{(k)}_1,\ldots, \lambda^{(k)}_{n}) \in (\Rpl)^n$ for~$k \in \N$, and~$\delta_{k,1}$ is the Kronecker's delta. It is understood that~$\lambda^{(0)}=(0,0,\ldots)$.
\end{prop}

\begin{remark}
	Recall that~$A^{(k)}$ is of size~$n \times L_k$. Thus, by noticing that~$\Tabl{Q}(A^{(k)})$ can have at most~$n$ rows, we see that~$\Lambda^{(k)}$ takes values in $\{ \lambda \in \Y|\ \ell(\lambda) \le n\} \subset \Y$, that is, $\Lambda^{(k)}_j =0$ for~$j > n$. Therefore, we can regard~$\Lambda^{(k)}$ and~$\widetilde{\Lambda}^{(k)}$ as random vectors in~$\R^{n}$.
\end{remark}

\begin{proof}  
  We use the following notation,
  \begin{equation}
    \label{eq:Mnqfrak}
    f_{q}^{(N)}\left(\lambda^{(1)}, \ldots, \lambda^{(q)}\right) \df \frac{1}{Z_\Schur(N)}s_{\lambda^{(q)}}(x(N)) \prod_{k=1}^q s_{\lambda^{(k)}/\lambda^{(k-1)}} (y^{(k)}(N)),
  \end{equation}
  where the normalizing factor~$Z_\Schur(N)$ is given by~\eqref{eq:norm_fact} with~$x_i$ and~$y_j$ from~\eqref{eq:geom_approx_exp}, and 
  \begin{equation}
    \label{eq:propprooffqdef}
    \begin{aligned}
      &f_{q}\left(\lambda^{(1)}, \ldots, \lambda^{(q)}\right) \df \frac{1}{{Z}_{n,q}} \det{\left( e^{-(i-1) \lambda_j^{(q)}}\right)}_{i,j=1}^n\\
      &\times \prod_{k=1}^q \det{\left(\frac{e^{(\nu_k+(n-1)\delta_{k,1})\lambda_j^{(k-1)}}}{e^{(\nu_k+\ell_k-1-(n-j)\delta_{k,1})\lambda_i^{(k)}}}\left(e^{\lambda_i^{(k)}} - e^{\lambda_j^{(k-1)}} \right)_+^{\ell_k-1-(n-1)\delta_{k,1}}\right)}_{i,j=1}^n,
    \end{aligned}
  \end{equation}
  with~$Z_{n,q}$ given by~\eqref{eq:propZnq}.

  Set
  \begin{equation}
    \label{eq:lambdahat}
    \widehat{\lambda}^{(k)}(N) = [N \lambda^{(k)}], \quad k=1, \ldots, q, 
  \end{equation}
  where~$[\cdot]$ is the integer part applied component-wise.

  We will find the asymptotics of~$f_{q}^{(N)}\left(\widehat{\lambda}^{(1)}(N), \ldots, \widehat{\lambda}^{(q)}(N)\right)$ as~$N \to \infty$. First, by the determinantal representation of Schur polynomials, e.g., see Macdonald~\cite{Macdonald}, one has 
  \begin{equation}
    \label{eq:ineq1}
    \frac{s_{\widehat{\lambda}^{(q)}(N)}(x(N))}{N^{\frac{n(n-1)}{2}}}  = \frac{\det{\left( (1-\frac{i-1}{N})^{[N \lambda_j^{(q)}] -j + n}\right)}_{i,j=1}^n}{\prod\limits_{i<j}^n(j-i)} \underset{N \to \infty}{\longrightarrow} \frac{\det{\left( e^{-(i-1) \lambda_j^{(q)}}\right)}_{i,j=1}^n}{\prod\limits_{i=1}^n \GammaF{i}}.
  \end{equation}
 Due to Hadamard's inequality, \eqref{eq:ineq1} holds uniformly for~$\lambda^{(k)}_i$, $k=1,\ldots,q$, $i=1, \ldots, n$ in compact subsets of~$\Rpl$.
  
  The following asymptotics can be found in Borodin, Gorin, and Strahov~\cite{BorodinGorinStrahov},
  \begin{equation}
    \label{eq:ineq2}
    \begin{aligned}
      \frac{s_{\widehat{\lambda}^{(k)}(N)/\widehat{\lambda}^{(k-1)}(N)} (y^{(k)}(N))}{N^{n(\ell_k-1)}} \underset{N \to \infty}{\longrightarrow} &\frac{1}{(\GammaF{\ell_k})^n} \prod \limits_{i=1}^n \frac{e^{-\nu_k\lambda_i^{(k)}}}{e^{-(\nu_k+\ell_k-1)\lambda_i^{(k-1)}}}\\
      &\det{\left(\left(e^{-\lambda_j^{(k-1)}}-e^{-\lambda_i^{(k)}} \right)_+^{\ell_k-1}\right)}_{i,j=1}^n,
    \end{aligned}
  \end{equation}
  \begin{equation}
    \label{eq:ineq3}
    \frac{s_{\widehat{\lambda}^{(1)}(N)}(y^{(1)}(N))}{N^{n \ell_1-\frac{n(n+1)}{2}}} \underset{N \to \infty}{\longrightarrow} \frac{\prod\limits_{i=1}^ne^{-\nu_1 \lambda_i^{(1)}}\left(1-e^{-\lambda_i^{(1)}}\right)^{\ell_1-n}}{\prod\limits_{j=1}^n \GammaF{\ell_1-n+j}}  \prod\limits_{1 \le i<j\le n}\left(e^{-\lambda_j^{(1)}}-e^{-\lambda_i^{(1)}}\right),
  \end{equation}
  where the limit is also uniform for~$\lambda^{(k)}_i$, $k=1,\ldots,q$, $i=1, \ldots, n$, in compact subsets of~$\Rpl$.  Note that cruder estimates for the left-hand side of~\eqref{eq:ineq2}, as in~\eqref{eq:ineq1}, are not sufficient.

  The asymptotics of the normalizing factor follows from~\eqref{eq:norm_fact} and~\eqref{eq:geom_approx_exp},
  \begin{equation}
    \label{eq:ineq4}
    \begin{aligned}
      \frac{N^{n L_q}}{Z_\Schur(N)} = &N^{n L_q}\prod_{\substack{1 \le i \le n\\1 \le k \le q\\1 \le j \le \ell_k}} \left(1-x_iy_j^{(k)}\right) = N^{n L_q}\prod_{\substack{1 \le i \le n\\1 \le k \le q\\1 \le j \le \ell_k}} \left(1-e^{-\frac{\nu_k+i+j-2}{N}}\right) \\
      &\underset{N \to \infty}{\longrightarrow} \prod_{\substack{1 \le i \le n\\1 \le k \le q\\1 \le j \le \ell_k}} (i+j+\nu_k-2) = \prod_{\substack{1 \le i \le n \\ 1 \le k \le q}} \frac{\GammaF{i+\ell_k+\nu_k-1}}{\GammaF{i+\nu_k-1}}.
    \end{aligned}
  \end{equation}

  Plugging~\eqref{eq:ineq1}--\eqref{eq:ineq4} into~\eqref{eq:Mnqfrak} and recalling~\eqref{eq:lambdahat}, we see that after simple algebraic manipulations, for all the~$\lambda^{(j)}_i \ge 0$, one has
  \begin{equation}
    \label{eq:Mfrak_fin_asympt}
    N^{q n}\, f_{q}^{(N)}\left(\widehat{\lambda}^{(1)}(N), \ldots, \widehat{\lambda}^{(q)}(N)\right) \underset{N \to \infty}{\longrightarrow} f_{q}\left({\lambda}^{(1)}, \ldots, {\lambda}^{(q)}\right).
  \end{equation}

  Since the limits in~\eqref{eq:ineq1}--\eqref{eq:ineq4} are uniform and the limiting expressions are continuous in the~$\lambda^{(k)}$, \eqref{eq:Mfrak_fin_asympt} also holds uniformly for the~$\lambda^{(k)}_i$ in compact subsets of~$\Rpl$.

  Now, from~\eqref{eq:Schur_proc} in Proposition~\ref{prop:Schur_proc}, we have
  \begin{equation}
    \prob{\frac{\Lambda^{(k)}(N)}{N} \preceq v^{(k)},\ k=1, \ldots, q} = \sum \limits_{\substack{ \lambda^{(1)}\preceq\ldots\preceq\lambda^{(q)}\\ \lambda^{(\ell)} \preceq N v^{(\ell)}\\ \ell=1,\ldots,q }} f_{q}^{(N)}\left(\lambda^{(1)}, \ldots, \lambda^{(q)}\right).
  \end{equation}

  	Simple estimates show that 
	\begin{equation}
	  \begin{aligned}
	  	S := &\left|\sum \limits_{\substack{\lambda^{(1)}\preceq\ldots\preceq\lambda^{(q)}\\ \lambda^{(\ell)} \preceq N v^{(\ell)}\\ \ell=1,\ldots,q }} \left(f_{q}^{(N)}\left(\lambda^{(1)}, \ldots, \lambda^{(q)}\right) - \frac{1}{N^{n q}}f_{q}\left(\frac{\lambda^{(1)}}{N}, \ldots, \frac{\lambda^{(q)}}{N}\right)\right)\right|\\
		&\le \frac{1}{N^{nq}} \sum \limits_{\substack{\lambda^{(1)}\preceq\ldots\preceq\lambda^{(q)}\\ \lambda^{(\ell)} \preceq N v^{(\ell)}\\\ell=1,\ldots,q }} \sup \limits_{\substack{\lambda^{(1)}\preceq\ldots\preceq\lambda^{(q)}\\ \lambda^{(\ell)} \preceq N v^{(\ell)}\\\ell=1,\ldots,q }}\left|N^{n q}\, f_{q}^{(N)}\left(\lambda^{(1)}, \ldots, \lambda^{(q)}\right) - f_{q}\left(\frac{\lambda^{(1)}}{N}, \ldots, \frac{\lambda^{(q)}}{N}\right)\right|.
	  \end{aligned}
  \end{equation}
  The formula~\eqref{eq:Mfrak_fin_asympt} implies
  \begin{equation}
    \begin{aligned}
      &S \le C(s^{(1)},\ldots,s^{(q)}) \sup \limits_{\substack{\lambda^{(1)}\preceq\ldots\preceq\lambda^{(q)}\\ \lambda^{(\ell)} \preceq v^{(\ell)}\\\ell=1,\ldots,q }}\left|N^{n q}\, f_{q}^{(N)}\left(\widehat{\lambda}^{(1)}(N), \ldots, \widehat{\lambda}^{(q)}(N)\right) - f_{q}\left(\lambda^{(1)}, \ldots, \lambda^{(q)}\right)\right| \quad\underset{N \to \infty}{\longrightarrow} 0,
    \end{aligned}
  \end{equation}
  where the last supremum is over all vectors~$\lambda^{(\ell)}$, not necessarily over those with integer components, and as before, the relation~$\preceq$ is understood component-wise.

  Therefore, we arrive at
  \begin{equation}
    \label{eq:lim_riemann_sum}
    \begin{aligned}
      &\lim_{N \to \infty} \prob{\frac{\Lambda^{(k)}(N)}{N} \preceq v^{(k)},\ k=1, \ldots, q} = \lim_{N \to \infty} \sum \limits_{\substack{\lambda^{(1)}\preceq\ldots\preceq\lambda^{(q)}\\ \lambda^{(\ell)} \preceq N v^{(\ell)}\\\ell=1,\ldots,q }} \frac{1}{N^{n q}}f_{q}\left(\frac{\lambda^{(1)}}{N}, \ldots, \frac{\lambda^{(q)}}{N}\right)\\
      &=\int\limits_{\substack{r^{(1)}\preceq\ldots\preceq r^{(q)}\\ r^{(\ell)} \preceq v^{(\ell)}\\ \ell=1,\ldots,q }} f_q\left(r^{(1)},\ldots,r^{(q)}\right)\, dr^{(1)}\cdots dr^{(q)},
    \end{aligned}
  \end{equation}
  where $dr^{(k)} = dr^{(k)}_1\cdots dr^{(k)}_{n}$ for~$k=1, \ldots, q$. The last identity in~\eqref{eq:lim_riemann_sum} follows since the sum is the Riemann sum approximation for the integral.
  
  Notice that
  \begin{equation}
    \label{eq:zerofq}
    f_q\left(r^{(1)},\ldots,r^{(q)}\right) = 0
  \end{equation}
  if~$r^{(1)}\preceq\ldots\preceq r^{(q)}$ is not satisfied, because in this case the product of determinants~\eqref{eq:propprooffqdef} vanishes. Indeed, introduce the matrix~$Q = \{Q_{i,j}\}$, where
  \begin{equation}
    Q_{i,j} = \frac{e^{(\nu_k+(n-1)\delta_{k,1})r_j^{(k-1)}}}{e^{(\nu_k+\ell_k-1-(n-j)\delta_{k,1})r_i^{(k)}}}\left(e^{r_i^{(k)}} - e^{r_j^{(k-1)}} \right)_+^{\ell_k-1-(n-1)\delta_{k,1}},\quad i,j=1, \ldots, n,
  \end{equation}
  and suppose that~$r^{(k)}_{j_0}< r^{(k-1)}_{j_0}$ for some~$j_0 \in\{1,\ldots,n\}$. Consequently, the left lower-corner block of size~$(n-j_0+1) \times j_0$ of~$Q$ is zero. That is,
  \begin{equation}
    Q_{i,j} =0, \quad i=j_0, j_0+1, \ldots, n;\ j = 1, \ldots, j_0.
  \end{equation}

  Recalling the classical definition of the determinant as a signed sum of products of matrix elements, exactly one chosen from each row and column, one sees due to the pigeonhole principle that at least one of the factors in each product has to be zero (the block of zeros is ``too large''). Thus,
  \begin{equation}
    \det{Q} = 0,
  \end{equation}
  and~\eqref{eq:zerofq} follows.

  Consequently, the formula~\eqref{eq:lim_riemann_sum} turns into
  \begin{equation}
      \lim_{N \to \infty} \prob{\frac{\Lambda^{(k)}(N)}{N} \preceq s^{(k)},\ k=1, \ldots, q}
      =\int\limits_{\substack{r^{(\ell)} \preceq v^{(\ell)}\\ \ell=1,\ldots,q }} f_q\left(r^{(1)},\ldots,r^{(q)}\right)\, dr^{(1)}\cdots dr^{(q)}, 
  \end{equation}
  for arbitrary~$q$. This is exactly the distribution function corresponding to~\eqref{eq:prop_Mq}.
\end{proof}

As it is shown in Borodin, Gorin, and Strahov~\cite{BorodinGorinStrahov}, the distribution~\eqref{eq:prop_Mq} is tightly related to the truncated-unitary product-matrix point process. 

Consider a logarithmic version of the process~$\Trunc_{n,\vec{\nu}, \vec{\ell}}$ on~$\N \times (0,1)$ from Section~\ref{subsect:trunc}, that is,
\begin{equation}
	\Trunc^{\log}_{n,\vec{\nu}, \vec{\ell}}= - \log{\Trunc_{n,\vec{\nu}, \vec{\ell}}},
\end{equation}
where the logarithm is applied pointwise to the second coordinate of the process only. Then, this new point process can be thought of as induced by the distribution~\eqref{eq:prop_Mq}, and the corresponding kernel is exactly~\eqref{ker_trunc_prod_log}. This is the link between Schur processes and product-matrix processes that we alluded to in the introduction.

\subsection{Proof of Theorem~\ref{thm21}}
\label{sect_34}
The main goal of this section is to tie together the directed last-passage percolation problem and Schur processes. The following fact is well-known, e.g., see Baik, Deift, and Suidan~\cite{BaikDeiftSuidanBook}.
\begin{thm}
  \label{thm:percvialambda1}
  Let~$A$ be an array of size~$n \times L$. Consider the last-passage percolation time~$\Time^{A}$ associated with this array. Let~$(\Tabl{P}(A),\Tabl{Q}(A))$ be semi-standard Young tableaux of shape~$\lambda$ produced by the RSK algorithm. Then,
  \begin{equation}
    \Time^{A} = \lambda_1.
  \end{equation}
\end{thm}

First, we are going to apply this theorem to the family of arrays~$A^{(k)}$, $k \in \N$, of geometric random variables from Section~\ref{subsect3.2}. In light of Proposition~\ref{prop:Schur_proc}, we obtain the following corollary.
\begin{cor}
  \label{cor:lpp_time}
  Let~$(\Time^\Geom(k),\ k \in \N)$ be the last-passage time stochastic process corresponding to the array of geometric random variables from Section~\ref{subsect3.2}. Then,
  \begin{equation}
    \begin{aligned}
      &\prob{\Time^\Geom(k) \le v_k,\ k=1,\ldots,q} = \prob{\Lambda^{(k)}_1 \le v_k,\ k=1,\ldots,q}\\
      &=\frac{1}{Z_{\Schur}} \sum \limits_{\substack{\lambda^{(\ell)}_1 \le v_{\ell}\\ \ell=1, \ldots, q}} s_{\lambda^{(q)}}(x) \prod_{k=1}^q s_{\lambda^{(k)}/\lambda^{(k-1)}} (y^{(k)}),
    \end{aligned}
  \end{equation}
  where~$\Lambda^{(k)}$ is defined in~\eqref{eq:Lambdaprocess}. By definition, it is understood that~$s_{s_{\lambda^{(k)}/\lambda^{(k-1)}}} = 0$ whenever~$\lambda^{(k-1)} \preceq \lambda^{(k)}$ is not satisfied.
\end{cor}

Passing to the limit, yields an analogous result for exponential arrays.
\begin{prop}
	Let~$(\Time(k),\ k \in \N)$ be the last-passage time stochastic process corresponding to the exponential array from Section~\ref{subsect3.3}. Then,
  \begin{equation}
    \label{eq:coreq1}
    \prob{\Time(k) \le v_k,\ k=1,\ldots,q} = \prob{\widetilde{\Lambda}^{(k)}_1 \le v_k,\ k=1,\ldots,q},
  \end{equation}
  where the process~$(\widetilde{\Lambda}^{(k)},\ k \in \N)$ is the same as in Proposition~\ref{prop:fidi_lambda_conv}. 
\end{prop}
\begin{proof}
	Clearly, the last-passage time~$\Time(k)$ is a continuous function of the elements of the array. Because of the continuous mapping theorem, we find
  \begin{equation}
    \prob{\Time(k) \le v_k,\ k=1,\ldots,q} = \lim_{N\to\infty} \prob{\frac{\Time^\Geom_N(k)}{N} \le v_k, k=1,\ldots,q}.
  \end{equation}
  Corollary~\ref{cor:lpp_time} together with Proposition~\ref{prop:fidi_lambda_conv} yield
  \begin{equation}
    \label{eq:corproofeq2}
    \begin{aligned}
      &\prob{\Time(k) \le v_k,\ k=1,\ldots,q} = \lim_{N\to\infty} \prob{\frac{\Lambda^{(k)}_1}{N} \le v_k,\ k=1,\ldots,q} \\
      &=\prob{\widetilde{\Lambda}^{(k)}_1 \le v_k,\ k=1,\ldots,q}.
    \end{aligned}
  \end{equation}
\end{proof}

We have all the ingredients to prove our first main result.
\begin{proof}[Proof of Theorem~\ref{thm21}]
	Due to Borodin, Gorin, and Strahov~\cite{BorodinGorinStrahov}, the distribution~\eqref{eq:prop_Mq} induces the point process~$\TruncLog_{n,\vec{\nu},\vec{\ell}}$, whose kernel is~\eqref{ker_trunc_prod_log}, as we pointed out at the end of the previous section. Thus, the probability on the right-hand side of~\eqref{eq:coreq1} can be regarded as the gap probability for~$\TruncLog_{n,\vec{\nu},\vec{\ell}}$. Then, the formulas~\eqref{lpp_findimkern_fredh} and~\eqref{eq:th1_fdredh} are nothing but well-known representations for the gap probability in terms of the kernel, e.g., see Baik, Deift, and Suidan~\cite{BaikDeiftSuidanBook}. This concludes the proof.
\end{proof}

\section{Proofs of Theorems~\ref{thm22} -- \ref{thm25}}
\label{sect_4}
\subsection{Auxiliary formulas and asymptotics}
\label{sect_41}
We start off by recalling the following basic facts about the gamma function, for which we refer the reader to \cite{DLMF}.
\begin{prop}
	\label{prop:A1}
	The following inequalities take place,
  \begin{align}
    \label{eq:app1}
    &|\Gamma(x+iy)| \le \Gamma(x), \quad &&x >0,\\
    \label{eq:app2}
    &\sqrt{2 \pi} x^{x-1/2} e^{-x} < \GammaF{x} <  \sqrt{2 \pi} x^{x-1/2} e^{-x} e^{1/(12 x)}, \quad &&x>0,\\
    \label{eq:app3}
    &\left|\frac{\GammaF{x}}{\GammaF{x + i y}}\right| \le e^{\frac{\pi |y|}{2}}, \quad &&x \ge \frac{1}{2}.
  \end{align}
\end{prop}

\begin{prop}
	\label{prop:A2}
	The following asymptotic formulas take place,
	\begin{equation}
		\label{eq:app4}
		|\GammaF{x + i y}| = \sqrt{2 \pi} |y|^{x - \frac{1}{2}} e ^{-\frac{\pi |y|}{2}} (1+o(1))	
	\end{equation}
	uniformly for bounded~$x \in \mathbb{R}$,
	\begin{equation}
		\label{lemma1_proof:eq0}
		\log{\GammaF{z}} = \left(z - \frac{1}{2}\right) \log{z} - z + \log{\sqrt{2 \pi}}+\frac{1}{12z} + O\left(\frac{1}{z^3}\right)
	\end{equation}
	as~$z \to \infty$, uniformly for~$|\arg{z}| \le \pi-\epsilon$, where~$\epsilon>0$.
\end{prop}

Our next goal is to state certain asymptotic properties of the gamma function in the form convenient for our further use. Set
\begin{equation}
	\label{eq:def_Fa}
	F_a(z) = \log{\GammaF{z+a}} - \log{\GammaF{a}} - z \log{a} + \frac{z}{2a}.
\end{equation}
\begin{lem}
	\label{lemma1}
	For every~$\delta\in(0,1)$ one has
	\begin{equation}
		\label{lemma1:eq1}
		\lim\limits_{a \to \infty} \sup\limits_{|z|\le a^{1-\delta}} \left|\frac{2aF_a(z)}{z^2}- 1 \right| = 0.
	\end{equation}
	and
	\begin{equation}
		\label{lemma1:eq2}
		\lim\limits_{a \to \infty} \sup\limits_{\substack{z \in \Omega \\ |z|\ge a^{1+\delta}}} \frac{1}{a}\left| F_a(z) - \left(z+a-\frac{1}{2}\right) \log{\frac{z}{a}} + z\left(1-\frac{1}{2a}\right) - a \right| = 0, 
	\end{equation}
	where~$\Omega = \mathbb{C} \setminus \{z \in \mathbb{C}|\ |\im{z}| \le \alpha |\re{z}|,\ \re{z} \le 0 \}$, $\alpha>0$, is the complex plane with a conic neighborhood of~$(-\infty,0)$ removed.
\end{lem}
\begin{remark}
	The ratio in~\eqref{lemma1:eq1} is extended by continuity, i.e., by definition~$\frac{2aF_a(z)}{z^2}\Big|_{z=0} = 1$.
\end{remark}
\begin{proof}
	Note that in the premise of the lemma~$a, z+a \to \infty$ and for sufficiently large~$a$ one has~$\max\{\arg{a}, \arg{(z+a)}\} \le \pi - \epsilon$ for some small~$\epsilon>0$. Hence, we can use the known asymptotics for~$\log{\GammaF{z}}$ from~\eqref{lemma1_proof:eq0}.

	Then,
	\begin{equation}
		\label{lemma1_proof:eq1}
		\begin{aligned}
			F_a(z) = &\left(z+a-\frac{1}{2}\right) \log{\frac{z+a}{a}} - z\left(1-\frac{1}{2a}\right) - \frac{1}{12a} + \frac{1}{12(a+z)}\\
			&+ O\left(\frac{1}{a^3}\right) + O\left(\frac{1}{(a+z)^3}\right).
		\end{aligned}
	\end{equation}
	To prove~\eqref{lemma1:eq1}, we need to expand the right-hand side of~\eqref{lemma1_proof:eq1} in a series about the point~$z/a$, which is justified since~$|z/a| \to 0$. Collecting the terms yields
	\begin{equation}
		F_a(z) = \frac{z^2}{2a}\big(1 + o_a(1)\big),\quad a\to \infty,
	\end{equation}
	where~$o_a(1)$ is uniform for~$|z| \le a^{1-\delta}$, and thus~\eqref{lemma1:eq1} holds.

	In the second case, $|a/z| \to 0$, and we need to expand about~$a/z$. This yields
	\begin{equation}
		F_a(z) = \left(z+a-\frac{1}{2}\right) \log{\frac{z}{a}} - z\left(1-\frac{1}{2a}\right) + a(1+o_a(1)), \quad a \to \infty,
	\end{equation}
	where~$o_a(1)$ is uniform for~$|z| \ge a^{1 + \delta}$, and thus~\eqref{lemma1:eq2} also holds.
\end{proof}

Throughout the rest of this section, we fix the contours~$S_\sigma$ and~$S_\zeta$ as in Fig.~\ref{contour3}, unless indicated otherwise. We also allow for shifting~$S_\sigma$ to the left and deforming~$S_\zeta^{(n)}$, as long as we do not cross any poles. The next lemma will be used extensively in the sections to come.

\begin{lem}
	\label{lemma44}
	Let~$\delta \in (0,1)$ and~$a, \nu, \ell$ be sufficiently large. Then,
	\begin{itemize}
		\item [(a)] For every~$z \in \mathbb{C}$
		\begin{equation}
			\label{cor2:eq1}
			a F_a(z) \underset{a \to \infty}{\longrightarrow} \frac{z^2}{2};
		\end{equation}
	\item [(b)] For sufficiently large~$a$ there exist constants~$C_1, C_2>0$ and~$\beta_1, \beta_2>0$ such that for all~$\sigma \in S_{\sigma}$ satisfying~$|\im{\sigma}| \le a^{1-\delta}$ one has
			\begin{equation}
				\label{cor2:eq2}
			C_1 e^{-\beta_1 |\im{\sigma}|^2} \le e^{a \re{F_a(\sigma)} } \le C_2 e^{-\beta_2 |\im{\sigma}|^2};
			\end{equation}
		\item [(c)] For sufficiently large~$a$ there exist constants~$C_1, C_2>0$ and~$\beta_1, \beta_2 >0$ such that for all~$\zeta \in S_{\zeta}^{(a)}$ satisfying~$|\re{\zeta}| \le a^{1-\delta}$ one has
			\begin{equation}
				\label{cor2:eq3}
				C_1 e^{-\beta_1 |\re{\zeta}|^2} \le e^{-a \re{F_a(\zeta)} } \le C_2 e^{-\beta_2 |\re{\zeta}|^2};
			\end{equation}
		\item [(d)] For sufficiently large~$a$ there exist constants~$\beta_1, \beta_2>0$ such that for all~$\sigma \in S_\sigma$ satisfying~$|\im{\sigma}| \ge a^{1+\delta}$ one has
			\begin{equation}
				\label{cor2:eq4}
				-\beta_1 |\im{\sigma}| \le \re{F_a(\sigma)} \le - \beta_2 |\im{\sigma}|;
			\end{equation}
		\item [(e)] If~$\nu = o(\ell)$, then for sufficiently large~$\nu$ and for all~$\sigma \in S_\sigma$ satisfying~$|\im{\sigma}| \ge (\nu + \ell)^{1+ \delta}$ one has
			\begin{equation}
				\label{cor2:eq5}
				\re{(F_{\nu}(\sigma)-F_{\nu+\ell}(\sigma))} \le - \ell \log{|\im{\sigma}|} + (\nu + \ell) \log{(\nu + \ell)};
			\end{equation}
		\item [(f)] For sufficiently large~$a$ and for all~$\sigma \in S_\sigma$ satisfying~$|\im{\sigma| \ge a^{1+\delta}}$ one has
	\begin{equation}
		\label{cor2:eq6}
		\re{\left(F_a(-\sigma) - \log \Gamma(-\sigma)\right)} \le a \log{|\im{\sigma}|}.
	\end{equation}

	\end{itemize}
\end{lem}

\begin{proof}
	The formula~\eqref{cor2:eq1} follows directly from~\eqref{lemma1:eq1}. To prove~\eqref{cor2:eq2} and~\eqref{cor2:eq3}, fix~$\delta_1 \in (0, \delta)$, and note that because of~\eqref{lemma1:eq1}, one can choose a small~$\epsilon>0$ such that for sufficiently large~$a$ and for all~$z$ such that~$|z| \le a^{1-\delta_1}$ one has
	\begin{equation}
		\left|aF_a(z) - \frac{z^2}{2}\right| < \epsilon \left| \frac{z^2}{2}\right|. 
	\end{equation}
	Thus, passing to the real part on the left-hand side of the inequality above, one has
	\begin{equation}
		\left|a\re{F_a(z)} - \frac{|\re{z}|^2-|\im{z}|^2}{2}\right| < \epsilon \frac{|\re{z}|^2+|\im{z}|^2}{2},
	\end{equation}
	and consequently
	\begin{equation}
		(1-\epsilon) |\re{z}|^2 - (1+ \epsilon) |\im{z}|^2 < 2a \re{F_a(z)} <  (1+\epsilon)|\re{z}|^2 - (1 - \epsilon) |\im{z}|^2. 
	\end{equation}
	Now, for~$z \in S_\sigma$, the real part~$\re{z}$ is constant and, for~$a$ large enough, $|\im{z}| \le a^{1-\delta}$ implies~$|z| \le a^{1-\delta_1}$, so we arrive at~\eqref{cor2:eq2}. On the other hand, for~$z \in S_\zeta^{(a)}$, the imaginary part~$\im{z}$ is bounded uniformly in~$a$ and, for~$a$ large enough, $|\re{z}| \le a^{1-\delta}$ implies~$|z| \le a^{1-\delta_1}$, so we arrive at~\eqref{cor2:eq3}.

	To establish~\eqref{cor2:eq4}, notice that
	\begin{equation}
		\label{cor2_proof:eq1}
		\re{\left(F_a(\sigma) -F_{a+\re{\sigma}}(i \im{\sigma})\right)} = F_a(\re{\sigma}).
	\end{equation}
	Since~$\re{\sigma}$ is constant, the latter expression converges to zero as~$a \to \infty$ by~\eqref{cor2:eq1}. Consequently, and since~$|\im{\sigma}| \ge a^{1+\delta}$ implies~$|\im{\sigma}| \ge (a + \re{\sigma})^{1+\delta}$, it is sufficient to prove
	\begin{equation}
		-\beta_1 |w| \le \re{F_a(iw)} \le - \beta_2 |w|, \quad |w| \ge a^{1+\delta},
	\end{equation}
	for sufficiently large~$a$.

	To that end, use~\eqref{lemma1:eq2} to choose a small~$\epsilon>0$ such that for sufficiently large~$a$ and for all~$w$ such that~$|w| \ge a^{1+\delta}$ one has
	\begin{equation}
		\left| F_a(iw) - \left(iw+a-\frac{1}{2}\right) \log{\frac{iw}{a}} + i w \left( 1 - \frac{1}{2a}\right) - a \right| < \epsilon a. 
	\end{equation}
	Hence, passing to the real part, one obtains
	\begin{equation}
		\left| \re{F_a(i w)} - \left(a-\frac{1}{2}\right) \log{\frac{|w|}{a}} + \frac{\pi |w|}{2} -a \right| < \epsilon a, 
	\end{equation}
and
	\begin{equation}
		\left(a-\frac{1}{2}\right) \log{\frac{|w|}{a}} - \frac{\pi |w|}{2} + a(1 - \epsilon) < \re{F_a(iw)} < \left(a-\frac{1}{2}\right) \log{\frac{|w|}{a}} - \frac{\pi |w|}{2} + a(1 + \epsilon).
	\end{equation}

	Rewrite the left- and the right-hand side of the expression as
	\begin{equation}
		\left(a-\frac{1}{2}\right) \log{\frac{|w|}{a}} - \frac{\pi |w|}{2} + a(1 \pm \epsilon) = -\frac{\pi |w|}{2} \left(1 - \frac{2a-1}{\pi|w|}\log{\frac{|w|}{a}} - \frac{2a(1 \pm \epsilon)}{\pi |w|}\right).
	\end{equation}
	Since~$|w| \ge a^{1+\delta}$, the expression in the parentheses converges to one, and the claim follows.

	To prove~\eqref{cor2:eq5}, we first consider~$z = iw$ with~$|w| \ge (\nu + \ell)^{1+\delta} > \nu^{1+\delta}$ and use~\eqref{lemma1:eq2} to write 
	\begin{equation}
		\begin{aligned}
			&\re{F_{\nu}(iw)} < \left(\nu-\frac{1}{2}\right) \log{\frac{|w|}{\nu}} - \frac{\pi |w|}{2} + \nu(1 + \epsilon),\\
			-&\re{F_{\nu + \ell}(iw)} < -\left(\nu + \ell - \frac{1}{2}\right) \log{\frac{|w|}{\nu + \ell}} + \frac{\pi |w|}{2} - (\nu + \ell)(1 - \epsilon),
		\end{aligned}
	\end{equation}
This implies
	\begin{equation}
		\begin{aligned}
			&\re{(F_{\nu}(iw) - F_{\nu + \ell}(iw))}< \left(\nu-\frac{1}{2}\right) \log{\frac{|w|}{\nu}} -\left(\nu + \ell - \frac{1}{2}\right) \log{\frac{|w|}{\nu + \ell}} + 2 \epsilon \nu - \ell (1 - \epsilon)\\
			&<- \ell \log{|w|} + \left(\nu + \ell - \frac{1}{2}\right) \log{(\nu + \ell)} - \left(\nu - \frac{1}{2}\right) \log{\nu} + 2 \epsilon \nu - \ell(1- \epsilon)\\
			&< - \ell \log{|w|} + (\nu + \ell) \log{(\nu + \ell)} - C \ell
		\end{aligned}
	\end{equation}
	for some constant~$C>0$ and~$\nu$ large enough. Now, observe that both~$F_\nu(\re{\sigma})$ and~$F_{\nu + \ell}(\re{\sigma})$ converge to zero and are certainly dominated by~$C \ell$. Finally, using~\eqref{cor2_proof:eq1} implies the claim.

	To prove~\eqref{cor2:eq6}, fix~$\delta_1 \in (0,\delta)$. Again, use~\eqref{lemma1:eq2} and write
	\begin{equation}
		\label{cor2_proof:eq2}
		\re{F_a(iw)} < \left(a-\frac{1}{2}\right) \log{\frac{|w|}{a}} - \frac{\pi |w|}{2} + a(1 + \epsilon),
	\end{equation}
	for some~$\epsilon > 0$ and~$a$ large enough, uniformly for~$|w| \ge a^{1+\delta_1}$. By choosing~$a$ even larger, if necessary, and by expanding the log-gamma function in a series by using~\eqref{lemma1_proof:eq0}, we find
	\begin{equation}
		\log{\GammaF{iw}} = \left(i w-\frac{1}{2}\right)\log{(iw)}-iw +O(1).
	\end{equation}
	Passing to the real part gives
	\begin{equation}
		\re{(\log{\GammaF{iw}})} = -\frac{1}{2} \log{|w|} - \frac{\pi |w|}{2}  + O(1),
	\end{equation}
and thus there exists~$C>0$ such that for sufficiently large~$|w|$ one has
	\begin{equation}
		\label{cor2_proof:eq3}
		-\re{(\log{\GammaF{iw}})} \le  \frac{1}{2} \log{|w|} + \frac{\pi |w|}{2}  + C.
	\end{equation}
	Adding together~\eqref{cor2_proof:eq2} and~\eqref{cor2_proof:eq3} yields 
	\begin{equation}
		\label{eq:4.35}
		\re{(F_a(iw)- \log{\GammaF{i w}})} \le a \log{|w|} - \left(a-\frac{1}{2}\right) \log{a} + a(1+\epsilon) +C. 
	\end{equation}

	By using~\eqref{lemma1_proof:eq0} and expanding in a series around~$\re{\sigma}/\im{\sigma}$, one can find, after passing to the real parts, that
	\begin{equation}
		\label{eq:4.36}
		\re{\left(\log{\GammaF{-\sigma}} - \log{\GammaF{-i \im{\sigma}}}\right)} = -\re{\sigma} \log{|\im{\sigma}|} + O(1).
	\end{equation}
	Further, it is readily verified that~\eqref{eq:4.35}, \eqref{eq:4.36}, and~\eqref{cor2_proof:eq1} yield
	\begin{equation}	
		\re{\left(F_a(-\sigma) - \log \Gamma(-\sigma)\right)} \le a \log{|\im{\sigma}|} - \left(a-\re{\sigma}-\frac{1}{2}\right) \log{a} + a(1+\epsilon) + \widetilde{C} 
	\end{equation}
	for some constant~$\widetilde{C}>0$ and~$a$ large enough. Further, we see that the negative term with~$a \log{a}$ dominates two last terms, and we arrive at the desired bound. Since~$|\im{\sigma}| \ge a^{1+\delta}$ implies~$|\im{\sigma}| \ge (a-\re{\sigma})^{1+\delta_1}$ for large enough~$a$, the proof is concluded.
\end{proof}

The following lemma is more subtle.

\begin{lem}
	\label{lemma45}
	If~$\nu = o(\ell)$, for every~$\epsilon>0$ one can choose the parameter~$\nu$ sufficiently large so that for all~$\sigma \in S_\sigma$ one has
	\begin{equation}
		\label{lemma3:eq1}
		\begin{aligned}
			2\nu\re{(F_{\nu}(\sigma) - F_{\nu + \ell}(\sigma))}  \le &(1+\epsilon) |\re{\sigma}|^2 - \nu |\im{\sigma}| \int \limits_{\frac{\nu+\re{\sigma}}{|\im{\sigma}|}}^{\frac{\nu + \ell+\re{\sigma}}{|\im{\sigma}|}} \frac{d \theta}{1+\theta^2} \\
			&= (1+\epsilon) |\re{\sigma}|^2 - \nu |\im{\sigma}| \int \limits_{\frac{|\im{\sigma}|}{\nu + \ell + \re{\sigma}}}^{\frac{|\im{\sigma}|}{\nu + \re{\sigma}}} \frac{d\theta}{1+\theta^2}. 
		\end{aligned}
	\end{equation}
\end{lem}
\begin{remark}
	Even though both integrals can be easily evaluated explicitly via $arctan$, we keep them as is for further convenience. 
\end{remark}

\begin{proof}
	First, similarly to~\eqref{cor2_proof:eq1},
	\begin{equation}
		\begin{aligned}
			&2 \nu\re{\left(F_{\nu}(\sigma) - F_{\nu + \ell}(\sigma)-F_{\nu+\re{\sigma}}(i \im{\sigma}) + F_{\nu + \ell+\re{\sigma}}(i \im{\sigma})\right)} \\
			&= 2\nu\left(  F_{\nu}(\re{\sigma}) - F_{\nu +\ell}(\re{\sigma}) \right) \to |\re{\sigma}|^2
		\end{aligned}
	\end{equation}
	as~$\nu \to \infty$ with~$\nu = o(\ell)$. In the last identity, we used~\eqref{cor2:eq1}.

	Therefore, we need only prove
	\begin{equation}
		\label{lemma3_proof:eq1}
		2\nu\re{(F_{\nu}(i w) - F_{\nu + \ell}(iw))}  \le  -\nu |w| \int \limits_{\frac{\nu}{|w|}}^{\frac{\nu + \ell}{|w|}} \frac{d \theta}{1+\theta^2} = -\nu |w| \int \limits_{\frac{|w|}{\nu + \ell}}^{\frac{|w|}{\nu}} \frac{d\theta}{1+\theta^2}, \quad w \in \mathbb{R}.
	\end{equation}

	To that end, write
	\begin{equation}
		2\re{(F_{\nu}(i w) - F_{\nu + \ell}(iw))} = -\sum \limits_{k=\nu}^{\nu+\ell-1} \log{\left(1 + \left( \frac{w}{k}\right)^2 \right)}.
	\end{equation}
	The inequality
	\begin{equation}
		\log{\frac{b}{a}} \ge 1 - \frac{a}{b}
	\end{equation}
	yields
	\begin{equation}
		I := \sum \limits_{k=\nu}^{\nu+\ell-1} \log{\left(1 + \left( \frac{w}{k}\right)^2 \right)} \ge w^2 \sum \limits_{k=\nu}^{\nu+\ell-1} \frac{1}{w^2 + k^2}. 
	\end{equation}
	Bounding each term from below by the integral,
	\begin{equation}
		I \ge w^2 \sum \limits_{k=\nu}^{\nu+\ell-1} \int \limits_k^{k+1} \frac{dx}{w^2+x^2} = w^2 \int \limits_{\nu}^{\nu + \ell} \frac{dx}{w^2+x^2},
	\end{equation}
	and changing the variable in the integral, we arrive at the first inequality in~\eqref{lemma3_proof:eq1}; the second bound follows by changing variables~$\theta \to \frac{1}{\theta}$. 
\end{proof}

We will also need the following straightforward lemmas.
\begin{lem}
	\label{lemma46}
	Let~$\sigma \in S_\sigma$. Then, there exists a constant~$C>0$ such that
	\begin{equation}
		\label{lemma4:eq1}
		\left|e^{F_n(-\sigma)}\right|=\left| \frac{e^{-\frac{\sigma}{2n}}\GammaF{n-\sigma}}{\GammaF{n} n^{-\sigma}}\right| \le C.
	\end{equation}
\end{lem}
\begin{proof}
	Using~\eqref{eq:app1} gives
	\begin{equation}
		\left| \frac{e^{-\frac{\sigma}{2n}}\GammaF{n-\sigma}}{n^{-\sigma}\GammaF{n} }\right| \le  \frac{e^{-\frac{\re{\sigma}}{2n}}\GammaF{n-\re{\sigma}} }{n^{-\re{\sigma}}\GammaF{n}} =e^{{F_n(-\re{\sigma})}}.
	\end{equation}
	Since~$F_n(-\re{\sigma}) \to 0$ as~$n \to \infty$ by~\eqref{cor2:eq1}, the claim follows.
\end{proof}

\begin{lem}
  \label{lemma47}
  Let~$\zeta \in S_\zeta^{(n)}$, where the contour is specified in Fig.~\ref{contour2}. Then, there exists a constant~$C>0$ independent of~$n$ such that
    \begin{equation}
	    \label{lemma5:eq1}
	    \left|e^{-F_n(-\zeta)}\right|=\left|\frac{\GammaF{n} e^{\frac{\zeta}{2n}}}{n^{\zeta}\GammaF{n-\zeta}}\right| \le C.
    \end{equation}
  \end{lem}
  \begin{proof}
	  Applying~\eqref{eq:app1} and~\eqref{eq:app3} from Proposition~\ref{prop:A1}, we have
    \begin{equation}
	    \left|\frac{\GammaF{n}e^{\frac{\zeta}{2n}}}{n^{\zeta}\GammaF{n-\zeta}}\right| \le e^{\frac{\pi |\im{\zeta}|}{2}} \frac{\GammaF{n} e^{ \frac{\re{\zeta}}{2n}}}{n^{\re{\zeta}} \GammaF{n-\re{\zeta}}}. 
    \end{equation}

    Using~\eqref{eq:app2} then yields
    \begin{equation}
      \begin{aligned}
	      &\frac{\GammaF{n} e^{ \frac{\re{\zeta}}{2n}}}{n^{\re{\zeta}} \GammaF{n-\re{\zeta}}} \le \widetilde{C} \left(1+\frac{\re{\xi}}{n-\re{\xi}} \right)^{n-\re{\xi}-1/2} e^{-\re{\xi}+\frac{1}{12n}} \le\widetilde{C} e^{-\frac{\re{\zeta}}{2(n-\re{\zeta})} + \frac{1}{12n}},
      \end{aligned}
    \end{equation}
    which implies the claim. 
  \end{proof}

  \begin{lem}
	  \label{lemma48}
	  Let~$S^h_\sigma$ be defined as in Fig.~\ref{contour2}. Then, there exists~$C>0$ such that for all~$\sigma \in S^h_\sigma$, $\re{\sigma} \le -1$, one has
	  \begin{equation}
		  \left|\frac{\GammaF{n-\sigma} n^\sigma}{\GammaF{n}\GammaF{-\sigma}}\right| \le C.
	  \end{equation}
  \end{lem}
  \begin{proof}
	  By using~\eqref{eq:app1} and~\eqref{eq:app3}, we can write
	  \begin{equation}
		  \left|\frac{\GammaF{n-\sigma} n^\sigma}{\GammaF{n}\GammaF{-\sigma}}\right| \le \frac{\GammaF{n-\re{\sigma}} n^{\re{\sigma}}}{\GammaF{n}\GammaF{-\re{\sigma}}} e^{\frac{\pi|\im{\sigma}|}{2}}.
	  \end{equation}
	  Now, set
	  \begin{equation}
		  \psi_n(x) = \frac{\GammaF{n+x}n^{-x}}{\GammaF{n} \GammaF{x}},
	  \end{equation}
	  and notice that due to Bernoulli's inequality one can write
	  \begin{equation}
		  \frac{\psi_{n+1}(x)}{\psi_n(x)} = \left(1+ \frac{x}{n}\right) \left(1 + \frac{1}{n}\right)^{-x} \le 1
	  \end{equation}
	  for~$x \ge 1$. The rest is to observe that
	  \begin{equation}
		  \psi_2(x) = \frac{x(x+1)}{2^x}
	  \end{equation}
	  is bounded for~$x \ge 1$.
  \end{proof}
\begin{lem}
	\label{lemma49}
	Let~$\ell,\nu \in \N$. Set $\sigma=-x+i\delta$, where $x>0$, and $\delta\neq 0$. Then, for large enough~$\ell$ there exist absolute constants~$M>0$ and~$\alpha>0$ that only depend on~$\delta$ such that
	\begin{equation}
		\label{eq:LemmaMeij1_eq1}
		\left|\frac{\Gamma(\nu + \ell)\Gamma(\nu + \sigma)}{\Gamma(\nu+\ell + \sigma)}\right|\leq M e^{\alpha \ell}.
	\end{equation}
\end{lem}
\begin{proof}
	Note that
\begin{equation}
	\label{MeijL11}
	\frac{\Gamma(\nu+ \sigma)}{\Gamma(\nu + \ell + \sigma)}=\frac{1}{(\nu + \ell -1 + \sigma)(\nu + \ell -2 + \sigma)\times \cdots \times (\nu+\sigma)}
\end{equation}
It is not difficult to see that if~$x\notin (\nu,\nu + \ell -1)$, then 
\begin{equation}
	\left|\frac{\Gamma(\nu+ \sigma)}{\Gamma(\nu + \ell + \sigma)}\right| \le \frac{M(\delta)}{(\ell-1)!} = \frac{M(\delta)}{\GammaF{\ell}},
\end{equation}
and if~$x \in (\nu,\nu + \ell -1)$, then
 \begin{equation}
	 \left|\frac{\Gamma(\nu+ \sigma)}{\Gamma(\nu + \ell + \sigma)}\right| \le \frac{M(\delta)}{([x] - \nu)! (\nu+\ell - [x] -2)!} = \frac{C_{\ell-2}^{[x]-\nu}\ M(\delta)}{\GammaF{\ell-1}}.
 \end{equation}

 Now, observe a simple inequality for binomial coefficients
 \begin{equation}
	 C_n^{k} \le C_n^{[\frac{n}{2}]} = \frac{\GammaF{n+1}}{(\GammaF{[n/2]+1})^2} \le e^{\alpha n}.
 \end{equation}
 The latter bound follows from~\eqref{eq:app2} for~$n$ large enough. Likewise, one has
 \begin{equation}
	 \frac{\GammaF{\nu+\ell}}{\GammaF{\ell}} \le e^{\alpha \ell}
 \end{equation}
 for~$\ell$ large enough. Altogether, we see that~\eqref{eq:LemmaMeij1_eq1} holds and we are done.
\end{proof}

  The following proposition is the key to analyzing uniform convergence of the kernels we will encounter further.
\begin{prop}
	\label{PropositionMeijFirth}
	Let~$A_j, B_j \subset \overline{\R}$, $j=1,2$, and $I \subset \overline{\R}$ be compacts, and suppose that
	\begin{equation}
		A_j = g(I \times B_j),\quad j=1,2,
	\end{equation}
	where~$g(\cdot, \cdot)$ is a continuous function.
	Let $\left\{\phi_n\right\}_{n \in \N}$ and $\left\{\psi_n\right\}_{n \in \N}$ be two sequences of functions continuous on~$A_1$ and~$A_2$ and converging uniformly to $\phi$ and $\psi$, respectively. Then,
	\begin{equation}
		K_n(x,y) \df \int_I \phi_n(g(u,x))\psi_n(g(u,y))\, m(du) \underset{n \to \infty}{\longrightarrow} K(x,y) \df \int_I \phi(g(u,x))\psi(g(u,y))\, m(du),
	\end{equation}
	uniformly for~$(x,y) \in B_1 \times B_2$, where~$m$ is a finite measure on~$I$.
\end{prop}
\begin{proof}
Write
\begin{equation}
	\label{InTh}
	\begin{aligned}
		\underset{\substack{x\in B_1\\y\in B_2}}{\sup} \left|K_n(x,y)-K(x,y)\right| \le &\underset{s \in A_2}{\sup} \left|\psi_n(s)-\psi(s)\right| \underset{s\in A_1}{\sup} \left|\phi_n(s)\right|\\
		&+ \underset{s\in A_2}{\sup}\left|\psi(s)\right| \underset{s \in A_1}{\sup}\left|\phi_n(s)-\phi(s)\right|.
	\end{aligned}
\end{equation}
Because of the uniform convergence, the limiting functions are continuous and the sequences are uniformly bounded. Hence, the right-hand side of~\eqref{InTh} converges to zero as~$n \to \infty$.
\end{proof}

\subsection{Proof of Theorem~\ref{thm22}}
\label{sect_42}
Further, we will often omit the argument~$\alpha$ in~$n(\alpha)$, $\nu(\alpha)$, and~$\ell(\alpha)$. Fix~$T>0$ such that $t_1, \ldots, t_N \in [0,T]$, and define
\begin{equation}
	g_\alpha(t) = \log{n} + [t \nu] \left(\log{\frac{\nu + \ell}{\nu}} + \frac{\ell}{2\nu (\nu + \ell)} \right)-\frac{1}{2n}.
\end{equation}
Set
\begin{equation}
	K_\alpha(q, x; r, y) \df K_{n,\vec{\nu},\vec{\ell}} (q, x; r, y)
\end{equation}
for the kernel~\eqref{ker_trunc_prod_log}, where~$n, \vec{\nu}$ and~$\vec{\ell}$ are as described in Section~\ref{Section2.2}.

One can plug in~\eqref{eq:thm1_transf} to rewrite the left-hand side of~\eqref{eq:thm1_fidi} in the form
\begin{equation}
	P \df \prob{\Crit_\alpha(t_k) \leq s_{k},\ k=1,\ldots,N} =  \prob{\Time_\alpha([t_k \nu]) \leq s_{k} + g_\alpha(t_k),\ k=1,\ldots,N}.
\end{equation}
The latter probability becomes a Fredholm determinant due to Theorem~\ref{thm21},
\begin{equation}
	P = \fdet{I-\chi_f K_\alpha  \chi_f}_{L_2(\{[t_1 \nu],\ldots,[t_N \nu]\}\times\Rpl)},
\end{equation}
where
\begin{equation}
	f(q,x)= \sum \limits_{j=1}^N \delta_{q,[t_j \nu]} \mathbf{1}_{(s_j+g_\alpha(t_j),+\infty)}(x). 
\end{equation}

By changing variables in the Fredholm determinant, one obtains
\begin{equation}
	\label{eq:thm22_proof_det}
	P = \fdet{I-\chi_{\hat{f}} \widehat{K}_\alpha  \chi_{\hat{f}}}_{L_2(\{t_1,\ldots,t_N\}\times\R, \mathfrak{c} \otimes m)},
\end{equation}
where~$\mathfrak{c}$ is the counting measure on~$\{t_1,\ldots,t_N\}$, $m(dx) = e^{-\kappa x}\ dx $,
\begin{equation}
	\label{eq:thm22_proof_eq1}
	\begin{aligned}
		\widehat{K}_\alpha(\tau, x; t, y) \df &e^{-x(T-\kappa-\tau)+y(T+\kappa-t)} \left(\frac{\GammaF{\nu + \ell}}{\GammaF{\nu}}\right)^{[t \nu] - [\tau \nu]} \\
	&\times K_\alpha ([\tau \nu], x + g_\alpha(\tau); [t \nu], y + g_\alpha(t)),\quad \kappa>0,
	\end{aligned}
\end{equation}
and
\begin{equation}
	\hat{f}(q,x)= \sum \limits_{j=1}^N \delta_{q,t_j} \mathbf{1}_{(s_j,+\infty)}(x). 
\end{equation}
Several commentaries are in order. Firstly, as easy to see, $g_\alpha(t) \to \infty$ as~$\alpha \to \infty$. Hence, above we assumed that~$\alpha$ is large enough so that~$s_j+g_\alpha(t_j)>0$ for all~$j=1,\ldots,N$. Since~$\hat{f}$ restricts the spacial coordinates in the kernel to~$(s_j,+\infty)$, our choice of~$\alpha$ justifies extending~$\Rpl$ to~$\R$ in the~$L_2$ space. Secondly, note that the kernel is conjugated by a gauge factor, which does not change the corresponding determinant. Lastly, we changed the measure in the Fredholm determinant, so the extra factor~$e^{\kappa(x+y)}$ appears in the kernel, and we choose~$\kappa$ to be
\begin{equation}
	\label{eq:thm22_proof_kappa}
	\kappa = \frac{1}{3} \min \{ |t-\tau| \in \R|\, t,\tau \in \{0, t_1, \ldots, t_N, T \}, t \ne \tau\} > 0.
\end{equation}
This form of the kernel is more convenient for our analysis as we can rely on Lemma~3.4.5 from Anderson, Guionnet, and Zeituni~\cite{AndersonGuionnetZeitouni}. This lemma implies that we need only prove the convergence of the kernel for fixed~$\tau$ and~$t$, uniformly for~$x \in [a_1,+\infty)$ and~$y \in [a_2,+\infty)$, where~$a_1, a_2 \in \R$ are arbitrary but fixed. 

We note that, instead of kernels, our further argument can be given in terms of trace-class operators and the~$L_2$ norms of the corresponding Hilbert--Schmidt factors in their decomposition. This, however, would have only complicated the matter without much added value.

Set 
\begin{equation}
	\label{eq:thm22_proof_h}
	h_{\alpha}(\eta; p) = \left( \frac{\GammaF{\eta + \nu}}{\GammaF{\nu}\nu^\eta}\frac{\GammaF{\nu+ \ell}(\nu + \ell)^\eta}{\GammaF{\eta + \nu + \ell}}\right)^{p} e^{\frac{p \ell \eta}{2\nu(\nu + \ell)}} =  e^{p (F_{\nu}(\eta) - F_{\nu + \ell}(\eta))}.
\end{equation}
Then, after straightforward manipulations and by using~\eqref{MeijerGFunction}, the kernel~\eqref{eq:thm22_proof_eq1} becomes
\begin{equation}
	\label{eq:scaled_kern}
	\begin{aligned}
		&\widehat{K}_{\alpha}(\tau, x; t, y) = -\frac{\mathbf{1}_{t>\tau}}{2 \pi i} \int \limits_{-\frac{1}{2} - i\infty}^{-\frac{1}{2}+i\infty} h_{\alpha}(\eta; [t \nu]-[\tau \nu]) e^{-x(\eta+T-\kappa-\tau)+y(\eta+T+\kappa-t)} \, d\eta \\
		&+ \int\limits_{S_{\sigma}}\frac{d\sigma}{2 \pi i}\int\limits_{S_{\zeta}^{(n)}}\frac{d\zeta}{2 \pi i}\ \frac{h_{\alpha}(\sigma; [t \nu])}{h_{\alpha}(\zeta; [\tau \nu])} \frac{e^{F_n(-\sigma)}\GammaF{-\zeta}}{e^{F_n(-\zeta)}\GammaF{-\sigma}} \frac{e^{-x(\zeta+T-\kappa-\tau)+y(\sigma+T+\kappa-t)}}{\sigma-\zeta},
	\end{aligned}
\end{equation}
where we write~$\mathbf{1}_{t>\tau}$ instead of~$\mathbf{1}_{[t \nu]> [\tau \nu]}$, which is justified for~$\alpha$ large enough. For convenience, the contours are presented in Fig.~\ref{contour6}.

Note that unlike in~\eqref{ker_trunc_prod_log}, where the~$\sigma$-integral is over~$S^h_\sigma$, in~\eqref{eq:scaled_kern} the corresponding integral is over~$S_\sigma$. Indeed, when all the arguments except for~$\sigma$ are fixed, one can write
\begin{equation}
	\left| h_{\alpha}(\sigma, [t \nu]) \frac{e^{F_n(-\sigma)}}{\GammaF{-\sigma}} e^{y\sigma} \right| \le \frac{C}{|\sigma|^{\ell [t \nu]-n}}.
\end{equation}
Since for~$\alpha$ large enough, $\ell [t \nu]-n > 1$, the contour~$S^h_\sigma$ can be deformed into a vertical straight line~$S_\sigma$. 

Additionally, note that, both in the single and in the double integral, the vertical line can be shifted arbitrarily, as long as (in the latter case) it does not intersect~$S_\zeta^{(n)}$ and the deformation does not cross the poles of the integrand. For the single integral this follows from the estimate 
\begin{equation}
	\left|h_{\alpha}(\eta; [t \nu]- [\tau \nu])\right| \le \frac{C}{|\eta|^{\ell( [t \nu] - [\tau \nu])}}
\end{equation}
and the fact that~$\ell( [t \nu] - [\tau \nu])>1$ for~$\alpha$ large enough. We will use this in our analysis below.
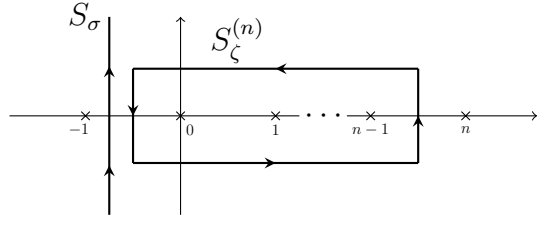
\begin{figure}[ht!]
	\centering
	\begin{tikzpicture}[scale=2.5]
		\begin{scope}[compass style/.style={color=black}, color=black, decoration={markings,mark= at position 0.5 with {\arrow{stealth}}}]

	% Axes
	% x-axis 
			\draw (-0.9,0) -- (0.625,0);
			\draw[->] (0.875,0) -- (1.875,0);

			\draw (0.765,0) node {$\cdots$};

		% crosses
			\draw (1.5,0) node [cross] {};
			\draw (1,0) node [cross] {};
			\draw (0.5,0) node [cross] {};
			\draw (0,0) node [cross] {};
			\draw (-0.5,0) node [cross] {};

		% nodes
			\node at (1.5,-0.075) {\scalebox{0.5}{$n$}};
			\node at (1.0,-0.075) {\scalebox{0.5}{$n-1$}};
			\node at (0.5,-0.075) {\scalebox{0.5}{$1$}};
			\node at (0.05,-0.075) {\scalebox{0.5}{$0$}};
			\node at (-0.535,-0.075) {\scalebox{0.5}{$-1$}};
	% y-axis
			\draw[->] (0,-0.525) -- (0,0.525);
	%\zeta-contour
			\draw[thick, postaction=decorate] (-0.25,0.25) -- (-0.25, -0.25);
			\draw[thick, postaction=decorate] (-0.25,-0.25) -- (1.25,-0.25);
			\draw[thick, postaction=decorate] (1.25,-0.25) -- (1.25, 0.25);
			\draw[thick, postaction=decorate] (1.25,0.25) -- (-0.25,0.25);

			%\draw[postaction=decorate] (0.23, -0.25) -- (0.23,0.25);
			%\draw[postaction=decorate] (0.27, 0.25) -- (0.27,-0.25);

			\node at (0.3,+0.4) {$S_{\zeta}^{(n)}$};
	%\sigma-contour
			\draw[thick, postaction=decorate] (-0.375,-0.525) -- (-0.375,0);
			\draw[thick, postaction=decorate] (-0.375,0) -- (-0.375,0.525);

			\node at (-0.5,0.525) {$S_{\sigma}$};
		\end{scope}
	\end{tikzpicture}  
	\caption{The contour~$S_{\sigma}$ and~$S_{\zeta}^{(n)}$.} 
	\label{contour6}
\end{figure}

Observe
\begin{equation}
\frac{1}{\sigma-\zeta} = -\int \limits_0^{+\infty} e^{u(\sigma - \zeta)}\, du, \quad \zeta \in S_\zeta^{(n)}, \ \sigma \in S_\sigma,
\end{equation}
and rewrite~\eqref{eq:scaled_kern} in the following form,
\begin{equation}
	\widehat{K}_{\alpha}(\tau, x; t, y) = \widehat{K}_{\alpha}^{(0)}(\tau, x; t, y) - \int \limits_0^{+\infty} \phi_{\alpha}(x + u) \psi_\alpha(y+u)\, du,
\end{equation}
where
\begin{align}
	&\widehat{K}_{\alpha}^{(0)}(\tau, x; t, y) = -\frac{\mathbf{1}_{t>\tau}}{2 \pi i} \int \limits_{S_\sigma} h_{\alpha}(\sigma; [t \nu]-[\tau \nu]) e^{-x(\sigma+T-\kappa-\tau)+y(\sigma+T+\kappa-t)} \, d\sigma, \\
\label{thm1_proof:eq1}
&\phi_\alpha(x) = \frac{e^{-x(T-\kappa-\tau)}}{2\pi i}\int\limits_{S_{\zeta}^{(n)}} (h_{\alpha}(\zeta, [\tau \nu]))^{-1} \frac{\GammaF{-\zeta}}{e^{F_n(-\zeta)}} e^{-x\zeta}\, d\zeta,\\
	\label{thm1_proof:eq2}
	&\psi_\alpha(y) = \frac{e^{y(T+\kappa-t)}}{2\pi i}\int\limits_{S_{\sigma}} h_{\alpha}(\sigma, [t \nu]) \frac{e^{F_n(-\sigma)}}{\GammaF{-\sigma}} e^{y\sigma}\, d\sigma.
\end{align}

Also, introduce
\begin{equation}
	\begin{aligned}
		\widehat{K}^{(0)}(\tau, x; t, y) &= -\frac{\mathbf{1}_{t>\tau}}{2 \pi i} \int \limits_{S_\sigma} e^{\frac{(t-\tau) \sigma^2}{2}-x(\sigma+T-\kappa-\tau)+y(\sigma+T+\kappa-t)} \, d\sigma\\
		&= -\frac{\mathbf{1}_{t>\tau}e^{-x(T-\kappa-\tau)+y(T+\kappa-t)}}{\sqrt{2 \pi (t-\tau)}} e^{-\frac{1}{2}\frac{(x-y)^2}{t-\tau}}
	\end{aligned}
\end{equation}
\begin{equation}
	\phi(x) = \frac{e^{-x(T-\kappa-\tau)}}{2\pi i}\int\limits_{S_{\zeta}}  \frac{\GammaF{-\zeta}}{e^{\frac{\tau \zeta^2}{2}}} e^{-x\zeta}\, d\zeta,\quad \psi(y) = \frac{e^{y(T+\kappa-t)}}{2\pi i}\int\limits_{S_{\sigma}}  \frac{e^{\frac{t \sigma^2}{2}}}{\GammaF{-\sigma}} e^{y\sigma}\, d\sigma,
\end{equation}
where the contours~$S_\sigma$ and~$S_\zeta$ are given in Fig.~\ref{contour3}. 

Observe that
\begin{equation}
	e^{-x(T-\kappa-\tau)+y(T+\kappa-t)}K_\Crit(\tau,x; t, y) = \widehat{K}^{(0)}(\tau, x; t, y) - \int \limits_0^{+\infty} \phi(x + u) \psi(y+u)\, du,
\end{equation}
where~$K_\Crit(\tau,x; t, y)$ is the extended critical kernel from~\eqref{eq:lim_kern} and the exponential gauge factor does not change the corresponding Fredholm determinant.

Na\"{i}vely, in view of~\eqref{cor2:eq1}, we expect that  
\begin{equation}
	\label{eq:thm24_proof_eq2}
	\begin{aligned}
		&\widehat{K}_{\alpha}^{(0)} (\tau,x; t,y) \to \widehat{K}^{(0)} (\tau,x; t,y),\\
		&\phi_{\alpha}(x) \to \phi(x),\quad \psi_{\alpha}(y) \to \psi(y),
	\end{aligned}
\end{equation}
as~$\alpha \to \infty$. Thus, by Proposition~\ref{PropositionMeijFirth} with~$I = [0,+\infty]$ and~$g(u,x) = x + u$, we have 
\begin{equation}
	\label{eq:thm24_proof_eq3}
	\widehat{K}_{\alpha}(\tau,x; t,y) \limalptoinf e^{-x(T-\kappa-\tau)+y(T+\kappa-t)}{K}_{\Crit} (\tau,x; t,y).
\end{equation}
Moreover, if the convergence in~\eqref{eq:thm24_proof_eq2} is uniform for~$x \in [a_1,+\infty)$ and~$y \in [a_2, +\infty)$, $a_1, a_2 \in \R$, Proposition~\ref{PropositionMeijFirth} will guarantee the uniform convergence in~\eqref{eq:thm24_proof_eq3}, which is what we seek.
\\[1ex]
\noindent \textbf{a)} We start by proving~$\psi_\alpha(y) \to \psi(y)$ uniformly for~$y \in [a_2,+\infty)$. First, deform the contour~$S_\sigma$ into~$-T + i \mathbb{R}$. The choice of~$\kappa$ in~\eqref{eq:thm22_proof_kappa} and simple bounds for the integral imply
\begin{equation}
	\label{thm1_proofa:eq1}
	\sup\limits_{y \in [a_2,+\infty)}|\psi_\alpha(y)- \psi(y)| \le C \int\limits_{S_\sigma} \left|h_{\alpha}(\sigma, [t \nu]) e^{F_n(-\sigma)}- e^{\frac{t \sigma^2}{2}}\right| \, \frac{|d\sigma|}{|\GammaF{-\sigma}|},
\end{equation}
where~$C>0$ is a constant. We will show that the right-hand side converges to zero by splitting~$S_\sigma$ into several pieces and analyzing them separately. Our analysis will heavily rely on Lemma~\ref{lemma44}.
\begin{enumerate}
	\item Suppose $|\im{\sigma}| \le \nu^{1-\delta}$, $\sigma \in S_\sigma$. We already know that due to~\eqref{cor2:eq1} the integrand in~\eqref{thm1_proofa:eq1} converges to zero pointwise. Thus, it is enough to establish an integrable bound and leverage the dominated convergence theorem.

		Recall~\eqref{eq:thm22_proof_h}, and then use~\eqref{cor2:eq2} from Lemma~\ref{lemma44} for~$F_\nu(\sigma)$ and~$F_{\nu + \ell}(\sigma)$, Lemma~\ref{lemma46} for~$F_n(-\sigma)$, and~\eqref{eq:app3} from Proposition~\ref{prop:A1} for~$\GammaF{-\sigma}$ to find that 
\begin{equation}
	\label{eq:thm22_proof_eq4}
	\left|h_{\alpha}(\sigma, [t \nu]) \frac{e^{F_n(-\sigma)}}{\GammaF{-\sigma}}\right| \le C e^{-\left(\beta_1 - \frac{\beta_2 \nu}{\nu + \ell}\right)|\im{\sigma}|^2 + \frac{\pi}{2}|\im{\sigma}|} \le \widetilde{C} e^{-\beta |\im{\sigma}|^2 + \frac{\pi}{2}|\im{\sigma}|}.
\end{equation}
The triangle inequality, the fact that the term with~$\sigma^2$ in~\eqref{thm1_proofa:eq1} is integrable, together with~\eqref{eq:thm22_proof_eq4}, supply the needed integrable bound and conclude this case.

\item Suppose that $\nu^{1+\delta} \le |\im{\sigma}| \le (\nu+\ell)^{1-\delta}$, $\sigma \in S_\sigma$.
	Since~$\nu \ll \ell$, we can always choose~$\delta>0$ small enough so that~$\nu^{1+\delta} < (\nu + \ell)^{1-\delta}$. Thus the inequality above does not collapse. 

	In this case, unlike that above, we will use~\eqref{cor2:eq4} for~$F_\nu(\sigma)$. For sufficiently large~$\alpha$, we find that 
\begin{equation}
	\left|h_{\alpha}(\sigma, [t \nu]) \frac{e^{F_n(-\sigma)}}{\GammaF{-\sigma}}\right| \le C e^{- \beta \nu |\im{\sigma}|\left(1 - \frac{|\im{\sigma}|}{\nu + \ell}\right)}e^{\frac{\pi|\im{\sigma}|}{2}} \le e^{-|\im{\sigma}|},
\end{equation}
which is again enough to apply the dominated convergence theorem.

	\item Suppose that~$ |\im{\sigma}| \ge (\nu + \ell)^{1+\delta}$, $\sigma \in S_\sigma$. Now, we cannot use~\eqref{cor2:eq2} for either~$F_{\nu + \ell}(\sigma)$ or~$F_{\nu}(\sigma)$. Moreover, \eqref{cor2:eq3}, together with Lemma~\ref{lemma46}, turns out to be insufficient. Nonetheless, we can still analyze this case by using a more direct calculation, instead of relying on the dominated convergence theorem. The goal is to show that
		\begin{equation}
			I_\alpha \df \int \limits_{|\im{\sigma}| \ge (\nu + \ell)^{1+\delta}} \left|h_{\alpha}(\sigma, [t \nu]) \frac{e^{F_n(-\sigma)}}{\GammaF{-\sigma}}\right| \, |d\sigma|
		\end{equation}
		vanishes in the limit~$\alpha \to \infty$. Use~\eqref{cor2:eq5} and~\eqref{cor2:eq6} from Lemma~\ref{lemma44} to write
\begin{equation}
	\left|h_{\alpha}(\sigma, [t \nu]) \frac{e^{F_n(-\sigma)}}{\GammaF{-\sigma}}\right| \le \myexp{-\nu \ell \left( \left(\beta-\frac{n}{\nu \ell}\right)\log{|\im{\sigma}|} - \beta \frac{\nu + \ell}{\ell} \log{(\nu + \ell)}\right)},
\end{equation}
for some~$\beta>0$.

Evaluating the integral, we have
\begin{equation}
	I_\alpha = 2 (\nu + \ell)^{\beta \nu (\nu + \ell)}\int \limits_{(\nu + \ell)^{1+\delta}}^{+\infty} \frac{dw}{w^{\beta \nu \ell - n}} =  \frac{2 (\nu + \ell)^{\beta \nu (\nu + \ell)}}{(\beta \nu \ell - n- 1)  (\nu + \ell)^{(1+\delta)(\beta \nu \ell - n-1)}}.
\end{equation}
The behavior of this expression depends on
\begin{equation}
	-(1+\delta)(\beta \nu \ell - n - 1) + \beta \nu (\nu + \ell) = - (\beta \nu \ell - n-1) \left(1 + \delta - \frac{\beta \nu (\nu + \ell)}{\beta \nu \ell -n-1}\right),
\end{equation}
which becomes negative as~$\alpha \to \infty$ since~$n \ll \nu \ll \ell$. This implies~$I_\alpha \limalptoinf 0$.

\item Suppose that~$\nu^{1-\delta} \le |\im{\sigma}| \le \nu^{1+\delta}$, $\sigma \in S_\sigma$. We will show that
\begin{equation}
	\label{thm1_proofa:eq2}
	\int \limits_{\nu^{1-\delta} \le |\im{\sigma}| \le \nu^{1+\delta}}  \left|h_{\alpha}(\sigma, [t \nu]) \frac{e^{F_n(-\sigma)}}{\GammaF{-\sigma}}\right|\, |d\sigma| \limalptoinf 0.
\end{equation}

Use Lemma~\ref{lemma45} and Lemma~\ref{lemma46} to find that
\begin{equation}
		A(\sigma) \df \left|h_{\alpha}(\sigma, [t \nu]) \frac{e^{F_n(-\sigma)}}{\GammaF{-\sigma}}\right|\le C \myexp{\frac{\pi |\im{\sigma}|}{2}-\beta \nu |\im{\sigma}| \int \limits_{\frac{|\im{\sigma}|}{\nu + \ell + \re{\sigma}}}^{\frac{|\im{\sigma}|}{\nu+\re{\sigma}}} \frac{d\theta}{1+\theta^2}}.
\end{equation}
Bounding~$\frac{1}{1+\theta^2}$ from below, one obtains
\begin{equation}
	A(\sigma) \le C \myexp{\frac{\pi |\im{\sigma}|}{2}- \frac{\beta \nu \ell |\im{\sigma}|^2}{(\nu + \ell +\re{\sigma})(\nu +\re{\sigma})} \frac{1}{1+{\frac{|\im{\sigma}|^2}{(\nu + \re{\sigma})^2}}}}. 
\end{equation}
Recalling that~$\nu^{1-\delta} \le |\im{\sigma}| \le \nu^{1+\delta}$, we arrive at
\begin{equation}
	A(\sigma) \le C \myexp{-\tilde{\beta} \nu^{1- 3\delta} |\im{\sigma}|} \le e^{-|\im{\sigma}|},
\end{equation}
where we chose~$\delta < \frac{1}{3}$ and $\alpha$ sufficiently large. This implies~\eqref{thm1_proofa:eq2}.

\item Suppose that~$(\nu + \ell)^{1-\delta} \le |\im{\sigma}| \le (\nu + \ell)^{1+\delta}$, $\sigma \in S_\sigma$. Now, we will show that
\begin{equation}
	I_\alpha \df \int \limits_{(\nu+\ell)^{1-\delta} \le |\im{\sigma}| \le (\nu+\ell)^{1+\delta}}  \left|h_{\alpha}(\sigma, [t \nu]) \frac{e^{F_n(-\sigma)}}{\GammaF{-\sigma}}\right|\, |d\sigma| \limalptoinf 0.
\end{equation}
To estimate~$F_{\nu}(\sigma)-F_{\nu+\ell}(\sigma)$ we again apply Lemma~\ref{lemma45}. Since~$n \ll \nu$, we can choose $\delta>0$ small enough so that~$n^{1+\delta} \le \nu^{1-\delta} \le (\nu + \ell)^{1-\delta}$. This enables us to use~\eqref{cor2:eq6} from Lemma~\ref{lemma44} for~$F_n(-\sigma)$. Proceed by writing

\begin{equation}
		\left|h_{\alpha}(\sigma, [t \nu]) \frac{e^{F_n(-\sigma)}}{\GammaF{-\sigma}}\right| \le \myexp{(n + \re{\sigma})\log{|\im{\sigma}|} -\beta \nu |\im{\sigma}| \int \limits_{\frac{\nu+\re{\sigma}}{|\im{\sigma}|}}^{\frac{\nu + \ell+\re{\sigma}}{|\im{\sigma}|}} \frac{d \theta}{1+\theta^2}}.
\end{equation}
Bounding~$\frac{1}{1+\theta^2}$ from below, we get
\begin{equation}
	\left|h_{\alpha}(\sigma, [t \nu]) \frac{e^{F_n(-\sigma)}}{\GammaF{-\sigma}}\right|  \le \myexp{(n + \re{\sigma})\log{|\im{\sigma}|}-\frac{\beta \nu \ell}{1+\frac{(\nu + \ell +\re{\sigma})^2}{|\im{\sigma}|^2}}}.
\end{equation}
Then, $(\nu + \ell)^{1-\delta} \le |\im{\sigma}| \le (\nu + \ell)^{1+\delta}$ gives us 
\begin{equation}
	\left|h_{\alpha}(\sigma, [t \nu]) \frac{e^{F_n(-\sigma)}}{\GammaF{-\sigma}}\right| \le e^{\beta_1 n \log{(\nu + \ell)}-\beta_2 \nu \ell^{1-2\delta}},
\end{equation}
if we choose~$\delta < \frac{1}{2}$ and $\alpha$ sufficiently large. By integrating this bound, we find that
\begin{equation}
	I_\alpha \le e^{(1+\delta+\beta_1 n) \log{(\nu + \ell)}-\beta_2 \nu \ell^{1-2\delta}} \limalptoinf 0
\end{equation}
since~$\nu \ell^{1 -2\delta}$ dominates.
\end{enumerate}

This finishes the proof of the claim for~$\psi_\alpha$.
\\[1ex]
\noindent \textbf{b)} Next step is to prove that~$\phi_\alpha(x) \to \phi(x)$ uniformly for~$x \in [a_1,+\infty)$. Deform the contour~$S_\zeta^{(n)}$ in such a way that~$\re{\zeta}\ge -\epsilon$, for some small~$\epsilon \in (0, \kappa)$. For convenience, set
\begin{equation}
	\widetilde{\phi}_{\alpha}(x) = \frac{e^{-x(T-\kappa-\tau)}}{2\pi i}\int\limits_{S_{\zeta}^{(n)}}  \frac{\GammaF{-\zeta}}{e^{\frac{\tau \zeta^2}{2}}} e^{-x\zeta}\, d\zeta.
\end{equation}

Then, we can write
\begin{equation}
	\sup\limits_{x \in [a_1,+\infty)}|\phi_\alpha(x) -\phi(x)| \le	\sup\limits_{x \in [a_1,+\infty)}|\phi_\alpha(x) -\widetilde{\phi}_\alpha(x)|  +  \sup\limits_{x \in [a_1,+\infty)}|\widetilde{\phi}_\alpha(x) - \phi(x)|.
\end{equation}
Because of the choice of~$\kappa$ and~$\epsilon$, and due to the fast decay of~$e^{-\frac{\tau \zeta^2}{2}} \GammaF{-\zeta}$, we see that 
\begin{equation}
	\sup\limits_{x \in [a_1,+\infty)}|\phi_\alpha(x) -\widetilde{\phi}_\alpha(x)| \limalptoinf 0.
\end{equation}
Hence, it is left to study the other term,

\begin{equation}
	\sup\limits_{x \in [a_1,+\infty)}|\widetilde{\phi}_\alpha(x) - \phi(x)| \le C \int\limits_{C_\zeta^{(n)}} \left|(h_{\alpha}(\zeta, [\tau \nu]))^{-1} e^{-F_n(- \zeta)} - e^{-\frac{\tau \zeta^2}{2}} \right|\cdot |\GammaF{-\zeta}|\, \cdot |d\zeta|. 
\end{equation}
We already know that, by~\eqref{cor2:eq1}, the integrand converges to zero pointwise. Again, we will rely on the dominant convergence theorem, for which it suffices to study
\begin{equation}
	A(\zeta) \df \left|(h_{\alpha}(\zeta, [\tau \nu]))^{-1}  \frac{\GammaF{-\zeta}}{e^{F_n(- \zeta)}} \right| = e^{-[\tau \nu] \re{(F_\nu(\zeta) - F_{\nu + \ell}(\zeta))}} \frac{|\GammaF{-\zeta}|}{e^{\re{F_n(-\zeta)}}}.
\end{equation}
Since~$n \ll \nu$, we can choose~$\delta>0$ small enough so that~$n \le \nu^{1-\delta} \le (\nu + \ell)^{1- \delta}$. Therefore, since~$|\re{\zeta}| \le n$, we can use~\eqref{cor2:eq3} from Lemma~\ref{lemma44}, together with Lemma~\ref{lemma47}, to obtain
\begin{equation}
	A(\zeta) \le C e^{-\beta \left(1-\frac{\nu}{\nu + \ell}\right)|\re{\zeta}|^2} \GammaF{- \zeta} \le C e^{-\beta |\re{\zeta}|^2} \GammaF{-\zeta},
\end{equation}
which holds for~$\alpha$ large enough. This supplies the sought integrable majorant and finishes the proof the claim.
\\[1ex]
\noindent \textbf{c)} The final goal is to establish the convergence 
\begin{equation}
	\sup\limits_{\substack{x \in [a_1,+\infty)\\y \in [a_2,+\infty)}}|\widehat{K}_{\alpha}^{(0)} (\tau,x; t,y) - \widehat{K}^{(0)} (\tau,x; t,y)| \limalptoinf 0.
\end{equation}

Set
\begin{equation}
	\mu = \frac{t + \tau}{2}.
\end{equation}
and deform the contour~$S_\sigma$ into~$-T + \mu + i \mathbb{R}$. This enables us to derive the following estimate,
\begin{equation}
	\label{eq:thm22_proof_eq_singint}
	\sup\limits_{\substack{x \in [a_1,+\infty)\\y \in [a_2,+\infty)}}|\widehat{K}_{\alpha}^{(0)} (\tau,x; t,y) - \widehat{K}^{(0)} (\tau,x; t,y)| \le C \int \limits_{S_\sigma} \left|h_\alpha(\sigma;[t\nu]-[\tau \nu]) - e^{\frac{(t - \tau)\sigma^2}{2}}\right|\, |d \sigma|
\end{equation}
because
\begin{equation}
	|e^{-x(\sigma + T -\kappa-\tau)+y(\sigma+T+\kappa-t)}|\Big|_{\sigma \in S_\sigma} = e^{-x(\mu-\kappa-\tau)+y(\mu+\kappa-t)} = e^{-(x+y)(\frac{t - \tau}{2}-\kappa)} \le 1.
\end{equation}
The pointwise convergence of the integrand follows from~\eqref{cor2:eq1}, and we can again take advantage of the dominated convergence theorem. The existence of the integrable majorant follows from the observation that the right-hand side of~\eqref{eq:thm22_proof_eq_singint} is similar to that of~\eqref{thm1_proofa:eq1}. In fact, the latter expression is more complicated because of the presence of~$\frac{F_n(-\sigma)}{\GammaF{-\sigma}}$. This means that all the estimates in~\textbf{a)} are going to apply to the present situation simply by erasing~$\frac{F_n(-\sigma)}{\GammaF{-\sigma}}$, which only makes the bounds stronger. This finishes the proof of the claim and concludes the proof of Theorem~\ref{thm22}.

\subsection{Proof of Theorem~\ref{thm23}}
\label{sect_43}

As in Kuijlaars and Zhang~\cite[Theorem~5.3]{KuijlaarsZhang}, we first notice that the pointwise convergence of the integrands takes place, and then will prove that the limit and the integrals can be interchanged. It is worth noting that unlike~\cite{KuijlaarsZhang} that uses the contour~$S_\sigma$, we use the contour~$S^h_\sigma$ instead, additionally providing details that were omitted in~\cite{KuijlaarsZhang}. To justify the interchange, we rely on lemmas from Section~\ref{sect_41}. We also note that the $s$-integral in~(5.7) of~\cite{KuijlaarsZhang} does not converge for~$M=1$. However, for~$M>1$ ($r>1$ in our notation) it does and one can always deform~$S^h_\sigma$ into~$S_\sigma$.

  After applying the scaling, the first term in~\eqref{TimeDependentGinibreKernel} becomes independent of~$n$. Hence, we will only analyze the second term. The complement formula for the gamma function implies 
\begin{equation}
	\label{eq:gamma_complement}
	\frac{\GammaF{\sigma+1} \GammaF{\zeta-n+1}}{\GammaF{\zeta+1}\GammaF{\sigma-n+1}} n^{\sigma - \zeta}= n^{\sigma - \zeta}\frac{\GammaF{n-\sigma} \GammaF{-\zeta}}{\GammaF{n-\zeta}\GammaF{-\sigma}} = e^{F_n(-\sigma) - F_n(-\zeta) + \frac{\sigma-\zeta}{2n}} \frac{\GammaF{-\zeta}}{\GammaF{-\sigma}}.
\end{equation}

The known asymptotics~\eqref{cor2:eq1} of the gamma function yields 
\begin{equation}
	e^{F_n(-\sigma) - F_n(-\zeta) + \frac{\sigma-\zeta}{2n}} \to 1,
\end{equation}
as~$n \to \infty$, pointwise. Hence, the claim will follow as soon as we establish that the limit can be exchanged with the double integral. 

To do so, we will use the dominated convergence theorem, and find an integrable bound. We recall that the~$\zeta$ integral in~\eqref{TimeDependentGinibreKernel} is over the contour~$S_\zeta^{(n)}$ in Fig.~\ref{contour2}. Thus, we can apply Lemma~\ref{lemma47} to bound~$e^{-F_n(-\zeta)}$.
\begin{figure}[ht!]
	\centering
	\begin{tikzpicture}[scale=2.5]
		\begin{scope}[compass style/.style={color=black}, color=black, decoration={markings,mark= at position 0.5 with {\arrow{stealth}}}]

	% Axes
	% x-axis 
			\draw[->] (-1.9,0) -- (0.625,0);

		% crosses
			\draw (-0.5,0) node [cross] {};
			\draw (-1,0) node [cross] {};
			\draw (-1.5,0) node [cross] {};

		% nodes
			\node at (0.05,-0.075) {\scalebox{0.5}{$0$}};
			\node at (-0.535,-0.075) {\scalebox{0.5}{$-1$}};
			\node at (-1.035,-0.075) {\scalebox{0.5}{$-2$}};
			\node at (-1.535,-0.075) {\scalebox{0.5}{$-3$}};

	% y-axis
			\draw[->] (0,-0.525) -- (0,0.525);
	%\sigma-contour
			\draw[postaction=decorate] (-0.77,-0.25) -- (-0.77,0.25);
			\draw[postaction=decorate] (-0.73,0.25) -- (-0.73,-0.25);
			\draw[thick, postaction=decorate] (-0.375,-0.25) -- (-0.375,0.25);
			\draw[thick, postaction=decorate] (-0.375,0.25) -- (-1.75,0.25);
			\draw[thick, postaction=decorate] (-1.75,-0.25) -- (-0.375,-0.25);

			\node at (-1,+0.4) {$S^h_{\sigma}$};

			\node at (-0.65,+0.17) {\scalebox{0.5}{$\gamma_0$}};
			\node at (-0.85,+0.17) {\scalebox{0.5}{$\gamma_1$}};
		\end{scope}
	\end{tikzpicture}  
	\caption{The contour~$S^h_{\sigma}$ decomposed in a loop~$\gamma_0$ around~$-1$ and a contour~$\gamma_1$ encircling~$-2,-3,\ldots$.}
	\label{contour5}
\end{figure}

Next, represent the contour~$S^h_\sigma$ as a sum of two contours, $\gamma_0$ and~$\gamma_1$, as indicated in Fig.~\ref{contour5}. Due to Lemma~\ref{lemma1}, we find that~$F_n(-\sigma)$ converges to zero uniformly on~$\gamma_0$ and thus is uniformly bounded. Then, certainly~$\frac{e^{F_n(-\sigma)}}{\GammaF{-\sigma}}$ is uniformly bounded on~$\gamma_0$, and the same holds on~$\gamma_1$ due to Lemma~\ref{lemma48}. The last piece of the argument is to observe that~$|x^\zeta y^{-\sigma-1}| \le e^{\alpha (|\re{\zeta}| + |\re{\sigma}|)}$ uniformly for~$x$ and~$y$ in compact subsets of~$\Rpl$, where~$\alpha>0$ is independent of either~$\zeta$ or~$\sigma$. The integrability of the majorant then follows from the fast convergence to zero of
\begin{equation}
	\frac{\prod \limits_{j=1}^{r} \GammaF{\sigma+\nu_j}}{\prod\limits_{j=1}^{q} \GammaF{\zeta+\nu_j}}
\end{equation}
as~$\sigma \to \infty$ or~$\zeta \to \infty$ along the corresponding contours, $S^h_\sigma$ or~$S_\zeta$. This concludes the proof.

 \subsection{Proof of Theorem~\ref{thm24}}
 \label{sect44}
 For the sake of readability, further we are going to omit the argument~$\alpha$ in~$n(\alpha)$ and~$\ell_j(\alpha)$.

 Set~$\widehat{K}_{\alpha}(q,x; r,y)$ to be the left-hand side of~\eqref{eq:thm_hard_eq1}. Recalling~\eqref{MeijerGFunction} and \eqref{eq:gamma_complement} together with a simple identity
 \begin{equation}
	 \frac{1}{\sigma - \zeta} = -\int \limits_0^1 u^{\zeta - \sigma - 1} \, du
 \end{equation}
 allows us to write
 \begin{equation}
 	\widehat{K}_{\alpha}(q,x; r,y)= \widehat{K}_{\alpha}^{(0)} (q,x; r,y)- \int \limits_0^1 \phi_\alpha (u x ) \psi_\alpha(u y) \, du,
 \end{equation}
where
\begin{align}
	\label{eq:4.63}
	&\widehat{K}_{\alpha}^{(0)} (q,x; r,y) =-\frac{\mathbf{1}_{r>q}}{2\pi i y} \int \limits_{S_\sigma} \prod_{j=q+1}^r \frac{\GammaF{\nu_j + \ell_j} (\nu_j + \ell_j)^\sigma}{\GammaF{\nu_j + \ell_j + \sigma}} \prod\limits_{j=q+1}^r \GammaF{\nu_j + \sigma} \left(\frac{x}{y}\right)^\sigma d\sigma,\\
	\label{eq:4.64}
	&\phi_{\alpha}(x) = \frac{1}{2\pi i} \int \limits_{S_\zeta^{(n)}}  \prod_{j=1}^q \frac{\GammaF{\nu_j + \ell_j + \zeta}}{\GammaF{\nu_j + \ell_j} (\nu_j + \ell_j)^\zeta} \frac{\GammaF{-\zeta}\GammaF{n}}{\GammaF{n- \zeta} n^\zeta} \frac{x^{\zeta} d\zeta}{\prod\limits_{j=1}^q \GammaF{\nu_j + \zeta}},\\
	\label{eq:4.65}
	&\psi_{\alpha}(y) = \frac{1}{2\pi i} \int \limits_{S^h_\sigma}\prod_{j=1}^r \frac{\GammaF{\nu_j + \ell_j} (\nu_j + \ell_j)^\sigma}{\GammaF{\nu_j + \ell_j + \sigma}} \frac{\GammaF{n-\sigma} n^\sigma}{\GammaF{n}\GammaF{-\sigma}} \prod\limits_{j=1}^r \GammaF{\nu_j + \sigma} y^{-\sigma - 1} d\sigma,
\end{align}
and the contours~$S^h_\sigma$ and~$S_\zeta^{(n)}$ are given in Fig.~\ref{contour2}. 
Introduce
\begin{equation}
	\widehat{K}^{(0)} (q,x; r,y) =-\frac{\mathbf{1}_{r>q}}{2\pi i y} \int \limits_{S^h_\sigma} \prod\limits_{j=q+1}^r \GammaF{\nu_j + \sigma} \left(\frac{x}{y}\right)^\sigma d\sigma,
\end{equation}
\begin{equation}
	\phi(x) = \frac{1}{2\pi i} \int \limits_{S_\zeta}  \frac{\GammaF{-\zeta}x^{\zeta} d\zeta}{\prod\limits_{j=1}^q \GammaF{\nu_j + \zeta}}, \quad \psi(y) = \frac{1}{2\pi i} \int \limits_{S^h_\sigma} \prod\limits_{j=1}^r \GammaF{\nu_j + \sigma}\frac{y^{-\sigma-1} d\sigma}{\GammaF{-\sigma}},
\end{equation}
where the integration contours are defined in Fig.~\ref{contour4}, and observe that
\begin{equation}
	K^{\Hardedge}_{\vec{\nu}}(q,x;r,y) = \widehat{K}^{(0)} (q,x; r,y) + \int\limits_0^1 \phi(ux) \psi(uy)\, du.
\end{equation}

Na\"{i}vely, in view of~\eqref{cor2:eq1}, we expect that  
\begin{equation}
	\label{eq:K0conv}
	\begin{aligned}
		&\widehat{K}_{\alpha}^{(0)} (q,x; r,y) \to \widehat{K}^{(0)} (q,x; r,y),\\
		&\phi_{\alpha}(x) \to \phi(x),\quad \psi_{\alpha}(y) \to \psi(y),
	\end{aligned}
\end{equation}
as~$\alpha \to \infty$, uniformly for~$x,y$ in compact subsets of~$\Rpl$. Thus, by Proposition~\ref{PropositionMeijFirth} with~$I = [0,1]$ and~$g(u,x) = ux$, 
\begin{equation}
	\widehat{K}_{\alpha}(q,x; r,y) \limalptoinf {K}^{\Hardedge}_{\vec{\nu}} (q,x; r,y)
\end{equation}
uniformly on compacts.

Again, we need only show that it is permissible to exchange the limits and the integrals for~\eqref{eq:4.63} -- \eqref{eq:4.65} and that the convergence is uniform.
\\[1ex]
\noindent \textbf{a)} We start with~$\widehat{K}_{\alpha}^{(0)} (q,x; r,y)$.

Since~$q$ and~$r$ are fixed, we can assume without loss of generality that
\begin{equation}
	N_1 \df \nu_{q+1}+\ell_{q+1} \le \ldots \le N_k \df \nu_{r} + \ell_{r}, \quad k = r-q.
\end{equation}

Decompose the negative half-axis in the following manner
\begin{equation}
	(-\infty,0]=(-1,0] \cup (-N_1,-1] \cup (-N_2,-N_1] \cup \ldots \cup (-N_{k},-N_{k-1}] \cup (-\infty, -N_k],
\end{equation}
and let~$S_0, S_1$, $S_2,\ldots, S_k$, and $S_{k+1}$ be the corresponding pieces of $S^h_\sigma$, see Fig.~\ref{contour5_1}.

On~$S_0$, the integrand converges uniformly, and thus it suffices to take care of the horizontal pieces~$S_j$, $j=1,\ldots,k+1$. 
\begin{figure}[ht!]
	\centering
	\begin{tikzpicture}[scale=2.5]
		\begin{scope}[compass style/.style={color=black}, color=black, decoration={markings,mark= at position 0.5 with {\arrow{stealth}}}]

	% Axes
	% x-axis 
			\draw (-1.9,0) -- (-1.1,0);
			\draw [densely dotted] (-1.05,0) -- (-0.9,0);
			\draw[->] (-0.85,0) -- (0.625,0);

		% nodes
			\node at (0.05,-0.075) {\scalebox{0.5}{$0$}};
			\draw (-0.5,0) node [cross] {};
			\node at (-0.5,-0.075) {\scalebox{0.5}{$-1$}};
			\node at (-0.75,-0.075) {\scalebox{0.5}{$-N_1$}};
			\node at (-1.2,-0.075) {\scalebox{0.5}{$-N_k$}};

	% y-axis
			\draw[->] (0,-0.525) -- (0,0.525);
	%\sigma-contour
			\draw[thick, postaction=decorate] (-0.125,-0.25) -- (-0.125,0.25);
			\draw[thick, postaction=decorate] (-0.125,0.25) -- (-0.5,0.25);
			\draw[thick] (-0.5,0.25) -- (-0.85,0.25);
			\draw[thick, postaction=decorate] (-1.1,0.25) -- (-1.75,0.25);
			\draw[thick, postaction=decorate] (-1.75,-0.25) -- (-1.1,-0.25);
			\draw[thick] (-0.85,-0.25) -- (-0.5,-0.25);
			\draw[thick, postaction=decorate] (-0.5,-0.25) -- (-0.125,-0.25);

    % Vertical separators (edges of intervals) marked with thick dots

			\foreach \x in { -0.5, -0.75, -1.2}{
			\draw[dotted] (\x,0.-0.5) -- (\x,0.5);
			\draw (\x,0.22) -- (\x,0.28);
			\draw (\x,-0.28) -- (\x,-0.22);

		}
		
		\node at (-0.25, 0.35) {\scalebox{0.75}{$S_0$}};
		\node at (-0.60, 0.35) {\scalebox{0.75}{$S_1$}};
		\node at (-1.5, 0.35) {\scalebox{0.75}{$S_{k+1}$}};

		\draw [densely dotted] (-1.05,0.25) -- (-0.9,0.25);
		\draw [densely dotted] (-1.05,-0.25) -- (-0.9,-0.25);

		\draw (-0.75,-0.03) -- (-0.75,0.03);
		\draw (-1.2,-0.03) -- (-1.2,0.03);

		\end{scope}
	\end{tikzpicture}  
	\caption{Splitting of the contour~$S^h_\sigma$ into~$S_j$, $j=0,\ldots,k+1$. Each~$S_j$, $j=1,\ldots,k+1$, is comprised of two disjoint horizontal intervals. The contour~$S_0$ is made of two horizontal intervals and the vertical segment.}
	\label{contour5_1}
\end{figure}
For convenience, set
\begin{equation}
	\begin{aligned}
		I_j= \left|\int \limits_{S_j} \prod_{p=q+1}^r \frac{\GammaF{\nu_p + \ell_p} (\nu_p + \ell_p)^\sigma}{\GammaF{\nu_p + \ell_p + \sigma}} \prod\limits_{p=q+1}^r \GammaF{\nu_p + \sigma} \left(\frac{x}{y}\right)^\sigma d\sigma\right|, \quad j = 1, \ldots, k+1. 
	\end{aligned}
\end{equation}

Consider~$S_1$. A simple change of variables shows that Lemma~\ref{lemma47} can be applied, so one gets
\begin{equation}
	\left|\prod\limits_{j=q+1}^r \frac{\GammaF{\nu_j + \ell_j} (\nu_j + \ell_j)^\sigma}{\GammaF{\nu_j + \ell_j + \sigma} }\right| \le C.
\end{equation}
The dominated convergence theorem then shows that
\begin{equation}
	\label{Meij3}
	-\frac{\mathbf{1}_{r>q}}{2\pi i y} \int \limits_{S_0 \cup S_1} \prod_{j=q+1}^r \frac{\GammaF{\nu_j + \ell_j} (\nu_j + \ell_j)^\sigma}{\GammaF{\nu_j + \ell_j + \sigma}} \prod\limits_{j=q+1}^r \GammaF{\nu_j + \sigma} \left(\frac{x}{y}\right)^\sigma d\sigma
\end{equation}
converges to~$\widehat{K}^{(0)} (q,x; r,y)$, uniformly for~$x$ and~$y$ in compact subsets of~$\Rpl$.

We need to show that the remaining integrals vanish. Consider~$S_{j}$, $2 \le j \le k$. We will use Lemma~\ref{lemma49} and Lemma~\ref{lemma47}, after a simple change of variables, to obtain
\begin{equation}
	\begin{aligned}
		&\left|\frac{\GammaF{\nu_p + \ell_p}\GammaF{\nu_p+\sigma}}{\GammaF{\nu_p + \ell_p + \sigma}}\right| \le C e^{\alpha \ell_p}, &&q + 1 \le p \le q + j-1, \\
		&\left|\frac{\GammaF{\nu_p+\ell_p}(\nu_p+\ell_p)^\sigma}{\GammaF{\nu_p + \ell_p + \sigma}}\right| \le C, &&q+ j \le p \le r.
	\end{aligned}
\end{equation}

Thus,
\begin{equation}
I_j \le C_1 e^{\alpha \sum \limits_{p= q+1}^{q+j-1} \ell_p} \int\limits_{N_{j-1}}^{N_j} \left(\prod\limits_{p=1}^{j-1} N_p\right)^{-s} \prod\limits_{p=q+j}^r \left|\GammaF{\nu_p - s + i \delta}\right| \left(\frac{y}{x}\right)^s ds.
\end{equation}

For~$s \in [N_{j-1},N_j]=[\nu_{q+j-1} + \ell_{q+j-1}, \nu_{q+j}+\ell_{q+j}]$ and~$p \ge q+j$, due to the Euler complement formula, elementary estimates from Proposition~\ref{prop:A1}, and monotonicity of the gamma-function with the large arguments, we have
\begin{equation}
	\left|\Gamma\left(\nu_p-s+i\delta\right)\right| \le \frac{C_2}{\GammaF{s-\nu_p+1}} \le \frac{C_2}{\GammaF{\nu_{q+j-1} + \ell_{q+j-1}-\nu_p}}.
\end{equation}
Hence, 
\begin{equation}
	I_j \le C   \frac{e^{\alpha \sum \limits_{p= q+1}^{q+j-1} \ell_p}}{\prod\limits_{p = q+j}^r\GammaF{N_{j-1}-\nu_p}}  \frac{e^{-N_{j-1}\left(\log(N_1\cdots N_{j-1}x)-\log{y}\right)}}{\log(N_1\cdots N_{j-1}x)-\log{y}}.
\end{equation}
Since there is at least one gamma-function in the denominator, it dominates the exponent in the enumerator, and the latter expression goes to zero as $N_1\rightarrow\infty$, $\ldots$, $N_{j-1}\rightarrow\infty$. 

The last step is to study~$I_{k+1}$. Use Lemma~\ref{lemma49} to find that
\begin{equation}
	I_{k+1} \le C  e^{\alpha \sum \limits_{p= q+1}^{r} \ell_p} \times \frac{e^{-N_{k}(\log\left(N_1\ldots N_{k}x\right) - \log{y})}}{\log\left(N_1\cdots N_{k}x\right)-\log{y}}. 
\end{equation}
This time the latter expression goes to zero as $N_1\rightarrow\infty$, $\ldots$, $N_{k}\rightarrow\infty$ because the exponent in the fraction dominates the other exponent.

Note that in all cases the converges is uniform in~$x$ and~$y$, as long as they are positive and separated from zero or infinity. The first line in~\eqref{eq:K0conv} is established.
\\[1ex]
\noindent \textbf{b)} The proof of~$\psi_{\alpha}(y) \to \psi(y)$ copies that of~$a)$ almost word for word. Indeed, observe the similarity between~\eqref{eq:4.63} and~\eqref{eq:4.65}, and we need only provide an extra bound for
\begin{equation}
	\frac{\GammaF{n-\sigma} n^\sigma}{ \GammaF{n}\GammaF{-\sigma}},
\end{equation}
which follows from Lemma~\ref{lemma48} if~$\re{\sigma} \le -1$ and from the uniform convergence for the rest of the contour.

The last step is to prove~$\phi_{\alpha}(x) \to \phi(x)$, with~$\phi_\alpha$ from~\eqref{eq:4.64}. Note that we can write
\begin{equation}
	\phi_{\alpha}(x) = \frac{1}{2\pi i} \int \limits_{S_\zeta^{(n)}} e^{-F_n(-\zeta) + \sum\limits_{j=1}^q F_{\nu_j + \ell_j}(\zeta) + \frac{\zeta}{2}\left(\frac{1}{n}-\sum\limits_{j=1}^q \frac{1}{\nu_j + \ell_j}\right)} \frac{\GammaF{-\zeta}x^{\zeta} d\zeta}{\prod\limits_{j=1}^q \GammaF{\nu_j + \zeta}},
\end{equation}
where we recall~$F_a(z)$ is defined in~\eqref{eq:def_Fa}. 

Since~$n \ll \ell_j$ for all~$j \in \N$, we can choose~$\delta \in (0,1)$ so that~$n \le (\nu + \ell_j)^{1- \delta}$. For~$\zeta \in S_\zeta^{(n)}$ one has~$|\re{\zeta}| \le n$, and we can use~\eqref{cor2:eq3}. This together with~\eqref{lemma5:eq1} yields
\begin{equation}
	\left|e^{-F_n(-\zeta) + \sum\limits_{j=1}^q F_{\nu_j + \ell_j}(\zeta) + \frac{\zeta}{2}\left(\frac{1}{n}-\sum\limits_{j=1}^q \frac{1}{\nu_j + \ell_j}\right)}\right| \le \widetilde{C} e^{|\re{\zeta}|^2\sum \limits_{j=1}^q \frac{\alpha}{\nu_j+\ell_j}} \le C e^{|\re{\zeta}|},
\end{equation}
where and~$n$ and the~$\ell_j$ are large enough. The dominated convergence theorem can be applied, and the proof is concluded.

\subsection{Proof of Theorem~\ref{thm25}}
 \label{sect_45}

 First, recall~\eqref{eq:thm22_proof_h}, and introduce a similar (but simpler) expression
\begin{equation}
	\widetilde{h}_{\nu}(\eta; p) = \left( \frac{\GammaF{\eta + \nu}}{\GammaF{\nu}\nu^\eta}\right)^{p} e^{\frac{p \eta}{2\nu}} =  e^{p F_{\nu}(\eta)},
\end{equation}
which is obtained from~\eqref{eq:thm22_proof_h} by erasing the dependence on~$\ell$ and~$n$. The kernel~$\widehat{K}_\nu(\tau,x; t, y)$ in~\eqref{eq:thm25_eq1} takes the form
\begin{equation}
	\label{eq:thm25_proof_eq1}
	\begin{aligned}
		&\widehat{K}_{\nu}(\tau, x; t, y) = -\frac{\mathbf{1}_{t>\tau}}{2 \pi i} \int \limits_{S^h_\sigma} h_{\nu}(\eta; [t \nu]-[\tau \nu]) e^{-(x-y)\eta} \, d\eta \\
		&+ \int\limits_{S_{\sigma}}\frac{d\sigma}{2 \pi i}\int\limits_{S_{\zeta}} \frac{d\zeta}{2 \pi i}\ \frac{h_{\nu}(\sigma; [t \nu])}{h_{\nu}(\zeta; [\tau \nu])} \frac{\GammaF{-\zeta}}{\GammaF{-\sigma}} \frac{e^{-x\zeta+y\sigma}}{\sigma-\zeta},
	\end{aligned}
\end{equation}
where~$S_\sigma$ and~$S_\zeta$ are as in Fig.~\ref{contour3}. Since~$\nu$ is large, we deformed $S^h_\sigma$ into~$S_\sigma$. Now, compare~\eqref{eq:thm25_proof_eq1} with~\eqref{eq:scaled_kern}. We see that overall, the former is a simplified version of the latter, with all the dependence on~$\ell$ or~$n$ erased. Then, the proof of Theorem~\ref{thm22} in Section~\ref{sect_42} goes through in the present case as well and is even simpler. In particular, we no longer need the condition~$n \ll \nu \ll \ell$. The artificial exponential gauge factors we used in \eqref{eq:scaled_kern} do not matter since~$x$ and~$y$ are in compacts of~$\R$. In fact, it is easy to extend the statement of the theorem to that of convergence of the corresponding gap probabilities. However, we do not pursue this here.

The rest of the details is left to the reader.

%%%%%%%%%%%%%%%%%%%%%%%%%%%%%%%%%%%%%%%%%%%%%%%%%%%%%%%%%%%%%%%%%%%%%%%%%%%%%%%%%%%%%%%%% 
%%%%%%%%%%%%%%%%%%%%%%%%%%%%%%%%%%%%%%%%%%%%%%%%%%%%%%%%%%%%%%%%%%%%%%%%%%%%%%%%%%%%%%%%% 
%%%%%%%%%%%%%%%%%%%%%%%%%%%%%%%%%%%%%%%%%%%%%%%%%%%%%%%%%%%%%%%%%%%%%%%%%%%%%%%%%%%%%%%%% 

\section*{Acknowledgements}
S.B. was supported by the Research Foundation -- Flanders (FWO), project 12K1823N ``New universal dynamical determinantal point processes associated with products of random matrices'' and in part by the European Research Council (ERC), Grant Agreement No.~101002013. E.S. was supported by the BSF grant 2018248 ``Products of random matrices via the theory of symmetric functions.'' The authors have no competing interests to declare that are relevant to the content of this article.

%%%%%%%%%%%%%%%%%%%%%%%%%%%%%%%%%%%%%%%%%%%%%%%%%%%%%%%%%%%%%%%%%%%%%%%%%%%%%%%%%%%%%%%%% 
%%%%%%%%%%%%%%%%%%%%%%%%%%%%%%%%%%%%%%%%%%%%%%%%%%%%%%%%%%%%%%%%%%%%%%%%%%%%%%%%%%%%%%%%% 
%%%%%%%%%%%%%%%%%%%%%%%%%%%%%%%%%%%%%%%%%%%%%%%%%%%%%%%%%%%%%%%%%%%%%%%%%%%%%%%%%%%%%%%%% 
%%%%%%%%%%%%%%%%%%%%%%%%%%%%%%%%%%%%%%%%%%%%%%%%%%%%%%%%%%%%%%%%%%%%%%%%%%%%%%%%%%%%%%%%% 


\begin{thebibliography}{19}
	\bibitem{Ahn1}{Ahn, A. Extremal singular values of random matrix products and Brownian motion on~$\mathrm{GL}(N,\mathbb{C})$. Probab. Theory Relat. Fields \textbf{187} (2023), 949--997.}

	\bibitem{Ahn2}{Ahn, A. Fluctuations of $\beta$-Jacobi product processes. Probab. Theory Relat. Fields \textbf{183} (2022), 57--123.}

\bibitem{Akemann1}{Akemann, G.; Burda, Z. Universal microscopic correlation functions for products of independent Ginibre matrices. J. Phys. A: Math. Theor. \textbf{45} (2012),  465201.}

\bibitem{AkemannKieburgWei}{Akemann, G.; Kieburg M.; Wei, L. Singular value correlation functions for products of Wishart random matrices. J. Phys. A. \textbf{46} (2013), 275205.}

\bibitem{AkemannIpsenKieburg}{Akemann, G.; Ipsen, J.; Kieburg M. Products of rectangular random matrices: singular values and progressive scattering.  Phys. Rev. E \textbf{88} (2013), 052118.}

\bibitem{AndersonGuionnetZeitouni}{Anderson, G. W.; Guionnet, A.; Zeitouni, O. An Introduction to Random Matrices. Cambridge studies in advanced mathematics, 118. Cambridge University press, 2010.}

\bibitem{BaikDeiftSuidanBook}{Baik, J.; Deift, P.; Suidan, T. Combinatorics and random matrix theory. Graduate Studies in Mathematics, 172. American Mathematicsal Society, Providence, RI, 2016.}

\bibitem{BorodinGorinStrahov}{Borodin, A.; Gorin, V.; Strahov, E. Product Matrix Processes as Limits of Random Plane Partitions, Int. Math. Res. Not. 2019.}

\bibitem{BorodinForrester}{Borodin, A.; Forrester, P. Increasing subsequences and the hard-to-soft edge transition in matrix ensembles. J. Phys. A: Math. Gen. \textbf{36} (2003), 2963--2981.}

\bibitem{BorodinPeche}{Borodin, A.; P\'{e}ch\'{e}, S. Airy kernel with two sets of parameters in directed percolation and random matrix theory. J. Stat. Phys. \textbf{132} (2008), no. 2, 275--290.}

\bibitem{Collins_phd}{Ben\^{o}\i{}t Collins. Int\'{e}grales matricielles et probabilit\'{e}s non-commutatives. Th\'{e}se de doctorat de l'Universit\'{e} Paris 6, 2003.}

\bibitem{CorwinHammond}{Corwin, I.; Hammond, A. Brownian Gibbs property for Airy line ensembles. Invent. math. \textbf{195} (2014), 441--508.}

\bibitem{JohansssonDiscretePolynuclearGrowth}{Johansson, K. Discrete polynuclear growth and determinantal processes. Comm. Math. Phys. \textbf{242} (2003), 277--329.}

\bibitem{JohanssonLastPassagePaper}{Johansson, K. On some special directed last-passage percolation models. Integrable systems and random matrices, 333--346, Contemp. Math., 458. Amer. Math. Soc., Providence, RI, 2008.}

\bibitem{JohanssonTwoTimes}{Johansson, K. The two-time distribution in geometric last-passage percolation. Probab. Theory Related Fields \textbf{175} (2019), no. 3--4, 849--895.}

\bibitem{JohanssonRahman}{Johansson, K.; Rahman, M. Multi-time distribution in discrete polynuclear growth. Comm. Pure Appl. Math. \textbf{74} (2021), no. 12, 2561--2627.}

\bibitem{KuijlaarsZhang}{Kuijlaars, A.B.J.; Zhang, L. Singular values of products of Ginibre random
  matrices, multiple orthogonal polynomials and hard edge scaling limits. Commun. Math. Phys. \textbf{332} (2014), 759--781.}

\bibitem{LiuWangWang}{Liu, D-Z.; Wang, D.; Wang, Y. Lyapunov Exponent, Universality and Phase Transition for Products of Random Matrices. Commun. Math. Phys. \textbf{399} (2023), 1811--1855.}

\bibitem{Luke} {Luke, Y.L. The special functions and their approximations. Academic Press, New York 1969.}

\bibitem{Macdonald}{Macdonald, I.~G. Symmetric functions and Hall polynomials. Oxford, 1995.}
\bibitem{Okounkov}{Okounkov, A. Infinite wedge and random partitions. Sel. math., New ser. \textbf{7} (2001), 57--81.}
\bibitem{OkounkovReshetikhin} {Okounkov, A.; Reshetikhin, N. Correlation Function of Schur Process with Application to Local Geometry of a Random 3-Dimensional Young Diagram. Journal of the AMS \textbf{16} (2003), no. 3, 581--603.}
\bibitem{PrahoferSpohn}{Pr\"{a}hofer, M.; Spohn, H. Scale Invariance of the PNG Droplet and the Airy Process. J. Statist. Phys. \textbf{108} (2002), 1071--1106.}

\bibitem{QuastelRemenik}{Quastel, J.; Remenik, D. Airy processes and variational problems. In "Topics in Percolative and Disordered Systems". Springer Proceedings in Mathematics and Statistics \textbf{69} (2014). Springer, New York, NY.}

\bibitem{Stanley}{Stanley, R.P. Enumerative Combinatorics. Volume 2. Cambridge University Press, 1999.}
\bibitem{StrahovD}{Strahov, E. Dynamical correlation functions for products of random matrices. Random Matrices Theory Appl. \textbf{4} (2015), no. 4, 1550020.}

\bibitem{DLMF}{NIST Digital Library of Mathematical Functions. \verb+http://dlmf.nist.gov+ (release 1.1.1 of 2021-03-15).}
\end{thebibliography}
\end{document}